\documentclass[review]{elsarticle}

\usepackage{lineno,hyperref}

\usepackage{framed,multirow}
\usepackage{verbatim}
\usepackage{amssymb}
\usepackage{latexsym}
\usepackage{amsmath}
\usepackage{lipsum}
\usepackage{amsfonts}
\usepackage{graphicx}
\usepackage{epstopdf}
\usepackage{bm,nicefrac}
\usepackage{subfig}
\usepackage{algorithm,algpseudocode,float}
\usepackage{setspace}
\usepackage{algcompatible}
\usepackage{tikz}
\usepackage{pgfplots}
\usepackage{mathptmx}
\usepackage{amsmath}
\usepackage{amssymb}
\usepackage{amsthm}

\newtheorem{thm}{Theorem}[section]
\newtheorem{lem}[thm]{Lemma}

\newtheorem{cor}{Corollary}

\theoremstyle{definition}

\usepackage{url}
\usepackage{xcolor}
\usepackage{comment}
\renewcommand{\hat}{\widehat}
\renewcommand{\tilde}{\widetilde}
\definecolor{newcolor}{rgb}{.8,.349,.1}
\newcommand{\logLogSlopeTriangle}[5]
{
	\pgfplotsextra
	{
		\pgfkeysgetvalue{/pgfplots/xmin}{\xmin}
		\pgfkeysgetvalue{/pgfplots/xmax}{\xmax}
		\pgfkeysgetvalue{/pgfplots/ymin}{\ymin}
		\pgfkeysgetvalue{/pgfplots/ymax}{\ymax}
		
		\pgfmathsetmacro{\xArel}{#1}
		\pgfmathsetmacro{\yArel}{#3}
		\pgfmathsetmacro{\xBrel}{#1-#2}
		\pgfmathsetmacro{\yBrel}{\yArel}
		\pgfmathsetmacro{\xCrel}{\xArel}
		
		\pgfmathsetmacro{\lnxB}{\xmin*(1-(#1-#2))+\xmax*(#1-#2)} 
		\pgfmathsetmacro{\lnxA}{\xmin*(1-#1)+\xmax*#1} 
		\pgfmathsetmacro{\lnyA}{\ymin*(1-#3)+\ymax*#3} 
		\pgfmathsetmacro{\lnyC}{\lnyA+#4*(\lnxA-\lnxB)}
		\pgfmathsetmacro{\yCrel}{\lnyC-\ymin)/(\ymax-\ymin)} 
		
		\coordinate (A) at (rel axis cs:\xArel,\yArel);
		\coordinate (B) at (rel axis cs:\xBrel,\yBrel);
		\coordinate (C) at (rel axis cs:\xCrel,\yCrel);
		
		\draw[#5]   (A)-- node[pos=0.5,anchor=north] {}
		(B)-- 
		(C)-- node[pos=0.5,anchor=west] {\tiny{#4}}
		cycle;
	}
}

\journal{Journal of Computational Physics}

\newcommand{\note}[1]{\textcolor{black}{#1}}
\newcommand{\LRp}[1]{\left(#1\right)}

\begin{document}


\begin{frontmatter}

\title{Bernstein-B\'ezier weight-adjusted discontinuous Galerkin methods for wave propagation in heterogeneous media}%

\author[1]{Kaihang {Guo}\corref{cor1}}
\author[1]{Jesse {Chan}}
\cortext[cor1]{Corresponding author: 
		Email: kaihang.guo@rice.edu; Tel.: +1-281-702-8829; }

\address[1]{Department of Computational and Applied Mathematics, Rice University, 6100 Main St, Houston, TX 77005, United States}

\begin{abstract}
This paper presents an efficient discontinuous Galerkin method to simulate wave propagation in heterogeneous media with sub-cell variations. This method is based on a weight-adjusted discontinuous Galerkin method (WADG), which achieves high order accuracy for arbitrary heterogeneous media \cite{chan2017weight}. However, the computational cost of WADG grows rapidly with the order of approximation. In this work, we propose a Bernstein-B\'ezier weight-adjusted discontinuous Galerkin method (BBWADG) to address this cost.  \note{By approximating sub-cell heterogeneities by a fixed degree polynomial, the main steps of WADG can be expressed as polynomial multiplication and $L^2$ projection, which we carry out using fast Bernstein algorithms.  The proposed approach reduces the overall computational complexity from $O(N^{2d})$ to $O(N^{d+1})$ in $d$ dimensions.  Numerical experiments illustrate the accuracy of the proposed approach, and computational experiments for a GPU implementation of BBWADG verify that this theoretical complexity is achieved in practice.}
\end{abstract}


\end{frontmatter}

 
\section{Introduction}
Efficient and accurate simulations of wave propagation are central to applications in seismology, where heterogeneities arise from the presence of different geological structures in the subsurface. Accurate and efficient numerical methods for wave problems are becoming more and more important as the demand for solutions of large-scale problems increases.  This paper presents an efficient discontinuous Galerkin (DG) method for wave equations in heterogeneous media with sub-cell variations.  
DG methods combine advantages of the finite volume method and the finite element method, which providing high order accuracy and addressing complex geometries through the use of unstructured meshes. These methods are straightforward to parallelize and can be accelerated by taking advantage of high performance architectures such as Graphics Processing Units (GPUs) \cite{klockner2009nodal}.

High order methods are especially attractive for wave propagation problems. The simulation of wave propagation is observed to be more robust to grid distortion at high orders than at low orders \cite{de2014connections,lorcher2008explicit}, and numerical dispersion and dissipation errors are small for high order approximations \cite{ainsworth2004dispersive}. The goal of this work is to address two issues related to high order DG methods for wave propagation: computational cost at high orders and accurate resolution of media with sub-cell heterogeneities.  Nodal DG methods, which are popular implementations of DG for wave propagation problems \cite{Hesthaven2007}, have a high computational complexity with respect to the order of approximation. We aim to reduce this computational complexity using Bernstein polynomials \cite{farouki2012bernstein}.

Bernstein polynomials have been previously utilized by Ainsworth at el.\ \cite{ainsworth2011bernstein} and Kirby \cite{kirby2011fast} to reduce computational costs associated with high order continuous  finite element methods on simplices.  More recent work has exploited properties of Bernstein polynomials for DG methods.  For example, Kirby introduced a fast algorithm in \cite{kirby2017fast} to invert the local mass matrix in DG schemes by exploiting a recursive block structure \note{present} under a Bernstein basis.  

\note{Chan and Warburton later introduced a Bernstein-B\'ezier discontinuous Galerkin (BBDG) method based on the ``strong'' DG formulation \cite{chan2017gpu}.  In contrast to the approach of Kirby \cite{kirby2017fast}, the use of the ``strong'' formulation avoids explicitly introducing a mass matrix inverse, and instead formulates the DG formulation in terms of differentiation and lifting matrices}.  BBDG exploits the facts that, in $d$ dimensions, \note{the derivative and the lift matrices can be recast as a combination of sparse matrices. By exploiting this structure}, the right-hand side of BBDG can be evaluated in $O(N^d)$ operations \note{per element}.  \note{In comparison, the dense linear algebra of} nodal DG methods generally \note{results in} a computational complexity of $O(N^{2d})$ \note{per element}. 

A separate challenge in the simulation of wave propagation is the \note{approximation of media heterogeneities}. High order finite difference methods are widely used \cite{virieux2011review} in practice, but face challenges for complex geometries and non-smooth media \cite{symes2009interface}. The spectral element method (SEM) \cite{komatitsch1998spectral} provides one alternative to explicit high order finite difference methods.  SEM produces a diagonal global mass matrix, making it well-suited for explicit time-stepping, and can accommodate both complex geometries (through unstructured meshes) and discontinuous media.  However, SEM is restricted to quadrilateral and hexahedral meshes, which are less geometrically flexible  than tetrahedral meshes. Several modifications have been proposed to extend SEM to triangular and tetrahedral meshes, but they require non-standard approximation spaces and do not support arbitrarily high order approximations \cite{chin1999higher}.  

An alternative to triangular and tetrahedral SEM are high order DG methods. High order DG methods can accommodate unstructured triangular and tetrahedral meshes, and naturally result in a block-diagonal global mass matrix, making them amenable to explicit time-stepping schemes and complex geometries. However, in most DG implementations for heterogeneous media, the discretization is based on the assumption that wavespeed is piecewise constant over each element \cite{chan2016gpu}. Fewer DG methods address the case when wavespeed varies within an element. Castro et al.\ \cite{castro2010seismic} addressed sub-element variations in wavespeed by recasting the wave equation as a new linear hyperbolic PDE with variable coefficients and source terms, which are non-zero in the presence of sub-element variations in wavespeed. However, this method introduces additional source terms and stiffness matrices with variable coefficients, resulting in a more complex formulation. Additionally, semi-discrete energy stability is not guaranteed. 

Mercerat and Glinsky \cite{mercerat2015nodal} proposed instead replacing the mass matrix by a weighted mass matrix, where the wavespeed acts as a weight function. The weighted mass matrix is obtained by introducing a set of quadrature points for the material approximation and computing integrals for entries of the mass matrix through quadrature rules. This modification does not require new stiffness matrices or source terms, and can be shown to be energy stable and high order accurate. However, because the wavespeed varies from element to element, each local weighted mass matrix is different. Thus, one needs to store inverses of weighted mass matrices over each element for time-explicit schemes, which significantly increases storage costs. Because GPUs have limited memory, these high storage costs restrict the problem sizes that can be run on a single GPU. Moreover, increased storage costs lead to more data movement, which is becoming increasingly expensive compared to the cost of floating point operations \cite{volkov2016understanding}. 

\note{To address these storage costs, we utilize a weight-adjusted approximation of the weighted mass matrix, whose inverse can be applied in a low-storage manner. The idea of a weight-adjusted approximation to a weighted mass matrix was first introduced as ``reverse numerical integration'' in \cite{koutschan2012computer}, though it was not analyzed in detail.  The idea was independently reintroduced and analyzed by Chan et al.\ in \cite{chan2017weight}.  The key idea is to approximate the weighted $L^2$ inner product using an equivalent weight-adjusted inner product, which produces provably high order accurate and energy stable DG methods with low storage requirements}.  Since WADG only modifies the local mass matrix, it maintains much of the structure of DG methods and is able to reuse existing DG implementations.  

The main \note{computational} step of WADG is the computation of a quadrature-based polynomial $L^2$ projection.  However, the implementation of the quadrature-based $L^2$ projection in WADG requires $O(N^{2d})$ operations, while complexity of BBDG is only $O(N^d)$. Hence, combining BBDG with WADG would result in the cost of the quadrature-based $L^2$ projection dominating the implementation at high polynomial degrees.  The goal of this work is to reduce the computational complexity of WADG at high orders of approximation, which we do using Bernstein bases. We develop an efficient algorithm to implement the polynomial $L^2$ projection in terms of Bernstein coefficients, which leads to a Bernstein-B\'ezier WADG (BBWADG) method. The main idea is to decompose the projection operator into a combination of degree elevation operators. Due to the sparsity of the one-degree elevation matrices, the $L^2$ projection can be applied in $O(N^{d+1})$ operations, reducing the complexity of right-hand evaluation from $O(N^6)$ to $O(N^4)$ in three dimensions. 

The paper is organized as follows. In Section~\ref{sec:WADG}, we review the weight-adjusted DG discretization of the first order acoustic and elastic wave equations in heterogeneous media. Section~\ref{sec:BBDG} introduces a Bernstein-B\'ezier DG method and its fast implementation. In Section~\ref{sec:BBWADG}, we propose a Bernstein-B\'ezier weight-adjusted DG method,  based on an algorithm to efficiently apply the polynomial $L^2$ projection under Bernstein bases. Section~\ref{sec:result} presents numerical validation and verification.

\section{Mathematical notation} 
\label{sec:notation}

In this paper, we focus on wave problems in three dimensions since BBWADG can reduce the computational complexity by two orders. In contrast, only one order of complexity can be reduced in two dimensions.

We assume the physical domain $\Omega$ is well approximated by a triangulation $\Omega_h$ consisting of $K$ non-overlapping elements $D^k$. The reference tetrahedron is defined as follows
$$\widehat{D}=\{\left( r,s,t\right)\geq -1; r+s+t\leq -1\}.$$
We assume that each element $D^k$ is the image of the reference element $\widehat{D}$ under an affine mapping $\bm{\Phi}^k$
$$\bm{x}=\bm{\Phi}^k\widehat{\bm{x}},\qquad \bm{x}\in D^k,\ \ \widehat{\bm{x}}\in \widehat{D},$$
where $\bm{x}=\left(x,y,z\right)$ are physical coordinates on the $k$th element and $\widehat{\bm{x}}=\left(r,s,t\right)$ are coordinates on the reference element. 
Over each element $D^k$, we define the approximation space $V_h\left(D^k\right)$ as
$$V_h\left(D^k\right)=\bm{\Phi}^k\circ V_h\left(\widehat{D}\right),$$
where $V_h\left(\widehat{D}\right)$ is a  polynomial approximation space of degree $N$ on the reference element. For the reference tetrahedron, $V_h\left(\widehat{D}\right)$ is defined as follows
$$V_h\left(\widehat{D}\right)=P^N\left(\widehat{D}\right)=\big\{r^i s^j t^k,\ \ 0\leq i+j+k\leq N\big\}.$$

In three dimensions, Bernstein polynomials on a tetrahedron are  expressed using barycentric coordinates. The barycentric coordinates for the reference tetrahedron are given as
$$\lambda_0=-\frac{\left(1+r+s+t\right)}{2},\ \ \ \lambda_1=\frac{\left(1+r\right)}{2},\ \ \ \lambda_2=\frac{\left(1+s\right)}{2},\ \ \ \lambda_3=\frac{\left(1+t\right)}{2}.$$
The $N$th degree Bernstein basis is simply as a scaling of the barycentric monomials
$$B^N_{ijkl}=C^N_{ijkl}\lambda_0^i\lambda_1^j\lambda_2^k\lambda_3^l,\qquad C^N_{ijkl}=\frac{N!}{i!j!k!l!},\qquad  i+j+k+l=N,$$ 
which forms a nonnegative partition of unity. For simplicity, we introduce the multi-index $\bm{\alpha}=\left(\alpha_0,\dots,\alpha_d\right)$ to denote the tuple of barycentric indices $\left(i,j,k,l\right)$. We define the order of a multi-index as
\begin{equation*}
	|\bm{\alpha}|:=\sum_{i=0}^{d}\alpha_i.
\end{equation*}
We take $\bm{\alpha}\leq \bm{\beta}$ to mean that $\alpha_j\leq \beta_j,\ \forall j=0,\dots,d$. 

\section{Weight-adjusted Discontinuous Galerkin methods}\label{sec:WADG} 

\note{The following sections} introduce weight-adjusted DG discretizations of acoustic and elastic wave equations.

\subsection{Acoustic wave equation}

We consider a first order velocity-pressure formulation of the acoustic wave equation given as
\begin{equation}
	\begin{split}
		\frac{1}{c^2}\frac{\partial p}{\partial \tau}+\nabla\cdot\bm{u}=0,\\
		\frac{\partial\bm{u}}{\partial \tau}+\nabla p = 0,
	\end{split}
\label{eq:awave}
\end{equation}
where $p$ is the acoustic pressure, $\bm{u}\in \mathbb{R}^d$ is the vector of velocities in each coordinate direction and $c$ is the wavespeed. We assume that (\ref{eq:awave}) is posed over time $\tau\in[0,T)$ on the physical domain $\Omega$ with boundary $\partial\Omega$, and the wavespeed is bounded by
$$0<c_{\textmd{min}}\leq c(\bm{x})\leq c_{\textmd{max}}<\infty.$$

We define the jump across element interfaces as
$$[\![p]\!]=p^+-p, \qquad [\![\bm{u}]\!]=\bm{u}^+-\bm{u},$$
where $p^+, \bm{u}^+$ and $p, \bm{u}$ are the neighboring and local traces of the solution over each interface, respectively. The average across an element interface is denoted by
$$\{\!\{p\}\!\}=\frac{1}{2}\left(p^++p\right),\qquad \{\!\{\bm{u}\}\!\}=\frac{1}{2}\left(\bm{u}^++\bm{u}\right).$$

We discretize the acoustic wave equation (\ref{eq:awave}) in space using a strong formulation and choose penalty fluxes as
$$\bm{u}^*=\{\!\{\bm{u}\}\!\}-\frac{\tau_p}{2}[\![p]\!]\bm{n},\qquad p^*=\{\!\{p\}\!\}-\frac{\tau_u}{2}[\![\bm{u}]\!]\cdot\bm{n},$$
where $\bm{n}$ is the outward unit normal vector on $D^k$. The corresponding semi-discrete formulation is given as follows
\begin{equation}
	\begin{split}
		\int_{D^k}\frac{1}{c^2}\frac{\partial p^k_h}{\partial \tau}\phi d\bm{x}&=-\int_{D^k}\nabla\cdot \bm{u}^k_h\phi d\bm{x}+\int_{\partial D^k} \frac{1}{2}\left(\tau_p[\![p_h^k]\!]-\bm{n}\cdot[\![\bm{u}^k_h]\!]\right)\phi d\bm{x},\\
		\int_{D^k}\frac{\partial \left(\bm{u}^k_h\right)_i}{\partial \tau}\bm{\psi}_i d\bm{x}&=-\int_{D^k} \frac{\partial p^k_h}{\partial \bm{x}_i}\bm{\psi}_i d\bm{x}+\int_{\partial D^k} \frac{1}{2}\left(\tau_u[\![\bm{u}^k_h]\!]\cdot\bm{n}-[\![p^k_h]\!]\right)\bm{\psi}_i\bm{n}_id\bm{x},
	\end{split}
\label{eq:sdf}
\end{equation}
where $\phi,\bm{\psi}$ are test functions and $\tau_p,\tau_u\geq 0$ are penalty parameters.

We define the mass matrix $\bm{M}$ and the face mass matrix $\bm{M}_f$ on $\widehat{D}$ as
$$\left(\bm{M}\right)_{ij}=\int_{\widehat{D}}\phi_i(\widehat{\bm{x}})\phi_j(\widehat{\bm{x}}) d\widehat{\bm{x}},\qquad \left(\bm{M}_f\right)_{ij}=\int_{f_{\widehat{D}}}\phi_i(\widehat{\bm{x}})\phi_j(\widehat{\bm{x}})d\widehat{\bm{x}}.$$
where $f_{\widehat{D}}$ is a face of the reference element $\widehat{D}$ and $\{\phi_i\}_{i=1}^{N_p}$ is an $N$th degree polynomial basis on $\widehat{D}$. Through an affine mapping $\bm{\Phi}^k$, we can map the local operators on $D^k$ to the reference operators
$$\bm{M}^k=J^k\bm{M},\qquad \bm{M}^{k}_f=J^{k}_f\bm{M}_f,$$
where $J^k$ is the determinant of the volume Jacobian and $J^{k}_f$ is the determinant of the face Jacobian for $f$. Similarly, the weighted mass matrix $\bm{M}^k_{w}$ on $D^k$ are given by
$$\left(\bm{M}^k_{w}\right)_{ij}=J^k\int_{\hat{D}}w\left(\bm{\Phi}^k\hat{\bm{x}}\right)\phi_i\left(\hat{\bm{x}}\right)\phi_j\left(\hat{\bm{x}}\right)d\hat{\bm{x}}.$$
The stiffness matrix on $D^k$ with respect to $x$ is defined as
$$\left(\bm{S}^k_x\right)_{ij}=\int_{D^k}\varphi_i\frac{\partial \varphi_j}{\partial x}d\bm{x},$$
and $\bm{S}^k_y, \bm{S}^k_z$ are defined similarly with respect to $y$ and $z$. Through chain rule, we can express stiffness matrices on $D^k$ in terms of the reference stiffness matrices $\bm{S}_1, \bm{S}_2, \bm{S}_3$ with respect to reference coordinates $r,s$ and $t$, respectively. Then, the semi-discrete formulation (\ref{eq:sdf}) can be written as
\begin{equation}
\begin{split}
\bm{M}^k_{1/c^2}\frac{d\bm{p}}{d\tau}&=-J^k\sum_{i=1}^{d}\left(\bm{G}^k_{i1}\bm{S}_1+\bm{G}^k_{i2}\bm{S}_2+\bm{G}^k_{i3}\bm{S}_3\right)\bm{U}_i+\sum_{f=1}^{N_{\textmd{faces}}}J^{k}_f\bm{M}_fF_p,\\
J^k\bm{M}\frac{d\bm{U}_i}{d\tau}&=-J^k\left(\bm{G}^k_{i1}\bm{S}_1+\bm{G}^k_{i2}\bm{S}_2+\bm{G}^k_{i3}\bm{S}_3\right)\bm{p}+\sum_{f=1}^{N_{\textmd{faces}}}J^{k}_f\bm{n}_i\bm{M}_fF_u,
\end{split}
\label{eq:matform}
\end{equation}
where $\bm{p}$ and $\bm{U}_i$ are degrees of freedom for $p$ and $\bm{u}_i$, and $\bm{G}^k$ is the matrix of geometric factors $r_x,s_x,t_x$, etc. The flux terms $F_p,F_u$ are defined such that
\begin{equation*}
\begin{split}
\left(\bm{M}_fF_p(\bm{p},\bm{p}^+,\bm{U},\bm{U}^+)\right)_j &= \int_{f_{\widehat{D}}}\frac{1}{2}\left(\tau_p [\![p]\!]-\bm{n}\cdot[\![\bm{u}]\!]\right)\phi_jd\hat{\bm{x}},\\
\left(\bm{n}_i\bm{M}_fF_u(\bm{p},\bm{p}^+,\bm{U},\bm{U}^+)\right)_j &= \int_{f_{\widehat{D}}}\frac{1}{2}\left(\tau_u [\![\bm{u}]\!]-[\![p]\!]\right)\bm{\psi}_j\bm{n}_id\hat{\bm{x}}.
\end{split}
\end{equation*}
Inverting $\bm{M}^k_{1/c^2}$ and $\bm{M}$ in (\ref{eq:matform}) produces a system of ODEs that can be solved by time-explicit methods. 

\note{When the wavespeed $c^2$ is approximated by a constant over each element, $\bm{M}^k_{1/c^2} = \frac{J^k}{c^2}\bm{M}$, and $\LRp{\bm{M}^k_{1/c^2} }^{-1} =  \frac{c^2}{J^k}\bm{M}^{-1}$.  Thus, to apply $\LRp{\bm{M}^k_{1/c^2} }^{-1}$, we need only store values of $J^k, c^2$ over each element and a single reference mass matrix inverse $\bm{M}^{-1}$ over the entire mesh.}\footnote{\note{In practice, the reference inverse mass matrix is incorporated into the definition of differentiation and lifting matrices on the reference element.}}. However, inverses of weighted mass matrices are \note{distinct from element to element if the wavespeed possesses sub-element variations.  Typical implementations precompute and store these weighted mass matrix inverese, which } significantly increases the storage cost of \note{high order} DG schemes.  

To address this issue, a weight-adjusted discontinuous Galerkin (WADG) is proposed in \cite{chan2018weight,chan2017weight}, which is energy stable and high order accurate for sufficiently regular weight functions. WADG approximates the weighted mass matrix by a weight-adjusted approximation $\widetilde{\bm{M}}^k_w$ given as
$$\bm{M}^k_w\approx\widetilde{\bm{M}}_w^k=\bm{M}^k\left(\bm{M}^k_{1/w}\right)^{-1}\bm{M}^k.$$
Plugging above expression into (\ref{eq:matform}),  we obtain the semi-discrete WADG discretization of (\ref{eq:awave}) as follows
\begin{equation}
	\begin{split}
	&\frac{d\bm{p}}{d\tau}=-\left(\bm{M}^k\right)^{-1}\bm{M}^k_{c^2}\left(\sum_{i=1}^{d}\sum_{j=1}^d\bm{G}_{ij}^k\bm{D}_j\bm{U}_i+\sum_{f=1}^{N_{\textmd{faces}}}\frac{J^k_f}{J^k}\bm{L}^fF_p\right),\\
	&\frac{d\bm{U}_i}{d\tau}=-\left(\bm{G}_{i1}^k\bm{D}_1+\bm{G}^k_{i2}\bm{D}_2+\bm{G}^k_{i3}\bm{D}_3\right)\bm{p}+\sum_{f=1}^{N_{\textmd{faces}}}\frac{J^k_f}{J^k}\bm{n}_i\bm{L}^fF_u,
	\end{split}
	\label{eq:WADGform}
	\end{equation}
\noindent where $\bm{D}_i=\bm{M}^{-1}\bm{S}_i$ are derivative operators with respect to reference coordinates $r,s,t$, $\bm{L}^f=\bm{M}^{-1}\bm{M}_f$ are lift operators over faces. 

\subsection{Elastic wave equation}

\note{The weight-adjusted approach can be extended to matrix-valued weights, which appear in symmetrized first order velocity-stress formulations of the elastic wave equation \cite{hughes1978classical}}.  Let $\rho$ be the density and $\bm{C}$ be the symmetric matrix form of constitutive tensor relating stress and strain. The first-order elastic wave equations are given by
\begin{equation}
\begin{split}
\rho\frac{\partial\bm{v}}{\partial \tau}&=\sum_{i=1}^{d}\bm{A}_i^T\frac{\partial\bm{\sigma}}{\partial\bm{x}_i},\\
\bm{C}^{-1}\frac{\partial\bm{\sigma}}{\partial \tau}&=\sum_{i=1}^{d} \bm{A}_i\frac{\partial\bm{v}}{\partial\bm{x}_i},
\end{split}
\label{eq:ewave}
\end{equation}
where $\bm{v}$ is the  velocity and $\bm{\sigma}$ is a vector consisting of unique entries of the symmetric stress tensor. The matrices $\bm{A}_i$ are given as
 $$\bm{A}_1=
 \renewcommand\arraystretch{1}
 \begin{pmatrix}
 1&0&0\\
 0&0&0\\
 0&0&0\\
 0&0&0\\
 0&0&1\\
 0&1&0
 \end{pmatrix},\qquad
 \bm{A}_2=
 \begin{pmatrix}
 0&0&0\\
 0&1&0\\
 0&0&0\\
 0&0&1\\
 0&0&0\\
 1&0&0
 \end{pmatrix},\qquad
 \bm{A}_3=
 \begin{pmatrix}
 0&0&0\\
 0&0&0\\
 0&0&1\\
 0&1&0\\
 1&0&0\\
 0&0&0
 \end{pmatrix}.
 $$
 For isotropic media, $\bm{C}$ is given by
 $$
 \bm{C}=
 \renewcommand\arraystretch{1}
 \setlength\arraycolsep{3pt}
 \begin{pmatrix}
 2\mu+\lambda &\lambda&\lambda&\\
 \lambda&2\mu+\lambda&\lambda&\\
 \lambda&\lambda&2\mu+\lambda&\\
 &&&&\mu\bm{I}^{3\times3}
 \end{pmatrix},$$
 where $\mu,\lambda$ are Lam\'e parameters. We note that $\bm{A}_i$ are \note{spatially constant independently of media heterogeneities.} 

\note{We can construct a semi-discrete DG scheme for elasticity analogous to the formulation (\ref{eq:sdf}) for the acoustic wave equation}
{\small
\begin{equation*}
\begin{split}
&\left(\rho\frac{\partial \bm{v}}{\partial \tau},\bm{w}\right)_{L^2(D^k)}=\left(\sum_{i=1}^{d}\bm{A}_i^T\frac{\partial\bm{\sigma}}{\partial\bm{x}_i},\bm{w}\right)_{L^2(D^k)}+\Bigg\langle\frac{1}{2}\bm{A}_n^T[\![\bm{\sigma}]\!]+\frac{\tau_v}{2}\bm{A}_n^T\bm{A}_n[\![\bm{v}]\!],\bm{w}\Bigg\rangle_{L^2(\partial D^k)},\\
&\left(\bm{C}^{-1}\frac{\partial \bm{\sigma}}{\partial \tau},\bm{q}\right)_{L^2(D^k)}=\left(\sum_{i=1}^{d}\bm{A}_i\frac{\partial\bm{v}}{\partial\bm{x}_i},\bm{q}\right)_{L^2(D^k)}+\Bigg\langle\frac{1}{2}\bm{A}_n[\![\bm{v}]\!]+\frac{\tau_{\sigma}}{2}\bm{A}_n\bm{A}_n^T[\![\bm{\sigma}]\!],\bm{q}\Bigg\rangle_{L^2(\partial D^k)},
\end{split}
\end{equation*}
}
where $(\cdot,\cdot)_{L(D^k)}$ and $\langle\cdot,\cdot\rangle_{L(D^k)}$ denote the $L^2$ inner product on $D^k$ and $\partial D^k$, respectively.  

\note{The presence of $\bm{C}^{-1}$ on the left-hand side produces a matrix-valued mass matrix $\bm{M}_{\bm{C}^{-1}}$ involving the constitutive stress tensor $\bm{C}$
\[
\bm{M}_{\bm{C}^{-1}} = \begin{bmatrix}
\bm{M}_{\bm{C}^{-1}_{11}} & \ldots &\bm{M}_{\bm{C}^{-1}_{1d}}\\
\vdots & \ddots & \vdots\\
\bm{M}_{\bm{C}^{-1}_{d1}} & \ldots &\bm{M}_{\bm{C}^{-1}_{dd}}\\
\end{bmatrix}, 
\]
where $\bm{C}^{-1}_{ij}$ denotes the $ij$th entry of $\bm{C}^{-1}$ and $\bm{M}_{\bm{C}^{-1}_{ij}}$ denotes the scalar weighted mass matrix with weight $\bm{C}^{-1}_{ij}$.   The matrix $\bm{M}_{\bm{C}^{-1}}$ can be understood as the matrix-weighted analogue of the scalar wavespeed-weighted mass matrix $\bm{M}_{1/c^2}$ which appeared for the acoustic wave equation.}

\note{
The inverse of $\bm{M}_{\bm{C}^{-1}}$ can be approximated by the inverse of a matrix-weighted weight-adjusted mass matrix 
\[
\bm{M}^{-1}_{\bm{C}^{-1}}\approx \LRp{\bm{I}\otimes \bm{M}^{-1}} \bm{M}_{\bm{C}} \LRp{\bm{I}\otimes \bm{M}^{-1}},
\]
where $\otimes$ denotes the Kronecker product.  We note that this approximation can be applied in terms of scalar weight-adjusted mass matrix inverses.  Incorporating this approximation yields the following}
WADG scheme for the elastic wave equations (\ref{eq:ewave})
 \begin{equation}
\begin{split}
\frac{\partial\bm{V}}{\partial \tau}& \!= \!\left( \!\bm{I} \!\otimes \!\left(\bm{M}^k\right)^{-1} \!\right) \!\bm{M}^k_{\rho^{-1} \!\bm{I}} \!\left(\sum_{i=1}^{d}\sum_{j=1}^{d}\bm{G}^k_{ij}\left(\bm{A}_i^T \!\otimes \! \bm{D}_j\right) \!\bm{\Sigma} \!+ \!\frac{J^k_f}{J^k}\sum_{f=1}^{N_{\textmd{faces}}} \!\left(\bm{I} \!\otimes \!\bm{L}^f\right) \!\bm{F}_v \!\right),\\
\frac{\partial\bm{\Sigma}}{\partial \tau}& \!= \!\left( \!\bm{I} \!\otimes \!\left(\bm{M}^k\right)^{-1} \!\right) \!\bm{M}^k_{\bm{C}} \!\left(\sum_{i=1}^{d}\sum_{j=1}^{d}\bm{G}^k_{ij}\left(\bm{A}_i \!\otimes \! \bm{D}_j\right) \!\bm{V} \!+ \!\frac{J^k_f}{J^k}\sum_{f=1}^{N_{\textmd{faces}}} \!\left(\bm{I} \!\otimes \!\bm{L}^f\right) \!\bm{F}_\sigma \!\right),
\end{split}
\label{eq:ewadg}
\end{equation}
where $\bm{V},\bm{\Sigma}$ are constructed by concatenating $\bm{\Sigma}_i,\bm{V}_i$ into single vectors, respectively, and  $\bm{F}_v, \bm{F}_{\sigma}$ are vectors representing the velocity and stress numerical fluxes.  \note{We note that this formulation is energy stable and high order accurate for elastic wave propagation in either isotropic or aniostropic heterogeneous media \cite{chan2018weight}.}

\subsection{Quadrature-based implementation}
\label{sec:qwadg}

In practice, weight-adjusted mass matrrix inverses are applied in a matrix-free fashion using sufficiently accurate quadrature rules.  \note{We follow \cite{chan2017weight}} and use simplicial quadratures which are exact for polynomials of degree $2N+1$ \cite{xiao2010numerical}.  Let $\widehat{\bm{x}}_i,\widehat{\bm{w}}_i$ denote the quadrature points and weights on the reference element. We define the interpolation matrix $\bm{V}_q$ as
$$\left(\bm{V}_q\right)_{ij}=\phi_j\left(\hat{\bm{x}}_i\right),$$
whose columns consist of values of basis functions at quadrature points. On each element $D^k$, we have 
\begin{equation*}
\bm{M}^k=J^k\bm{M}=J^k\bm{V}_q^T\textmd{diag}\left(\hat{\bm{w}}\right)\bm{V}_q,\ \ \ 
\bm{M}^k_{c^2}=J^k\bm{V}_q^T\textmd{diag}\left(\bm{d}\right)\bm{V}_q,\ \ \  \bm{d}_i=\frac{\hat{\bm{w}}_i}{c^2\left(\bm{\Phi}^k\hat{\bm{x}}_i\right)}
\end{equation*}
where $\bm{\Phi}^k\hat{\bm{x}}_i$ are quadrature points on $D^k$ and \note{ $c^2\left(\bm{\Phi}^k\hat{\bm{x}}\right)$ denote the values of the wavespeed at quadrature points}.  \note{Evaluating the right hand side of (\ref{eq:WADGform}) and (\ref{eq:ewadg}) requires applying the product of an unweighted mass matrix and weighted mass matrix, such as $\left(\bm{M}^k\right)^{-1}\bm{M}^k_{c^2}$.}  This can be done using quadrature-based matrices as follows
\begin{equation}
\left(\bm{M}^k\right)^{-1}\bm{M}^k_{c^2} = \bm{P}_q\textmd{diag}\left(\frac{1}{c^2\left(\bm{\Phi}^k\hat{\bm{x}}\right)}\right)\bm{V}_q,
\label{eq:pwadg}
\end{equation}
where $\bm{P}_q=\bm{M}^{-1}\bm{V}_q^T\textmd{diag}\left(\hat{\bm{w}}\right)$ is a quadrature-based polynomial $L^2$ projection operator on the reference element.  Moreover, since $\bm{P}_q, \bm{V}_q$ are reference operators, the implementation of (\ref{eq:pwadg}) requires only $O\left(N^d\right)$ storage for values of the wavespeed \note{$c^2\left(\bm{\Phi}^k\hat{\bm{x}}\right)$} at quadrature points for each element. In contrast, \note{storing full weighted mass matrix inverses or factorizations} requires $O\left(N^{2d}\right)$ storage on each element.  For example, in three dimensions, the number of quadrature points on one element, scales with $O(N_p)=O(N^3)$, while size of the weighted mass matrix inverse is $O(N_p)\times O(N_p)$, implying an $O(N^6)$ storage requirement.

\section{Bernstein-B\'ezier DG methods}\label{sec:BBDG} 

In this section, we review how to use Bernstein-B\'ezier polynomial bases to construct efficient high order DG methods.  \note{For nodal DG methods, the numerical fluxes can be computed in terms of the difference} between \note{degrees of freedom} at face nodes on two neighboring elements.  \note{This is also true of} the Bernstein basis, since Bernstein polynomials share a geometrical decomposition with vertex, edge, face and interior nodes in the sense that edge basis functions vanish at vertices, face basis functions vanish at vertices and edges, and interior basis functions vanish at vertices, edges, and faces \cite{farouki2012bernstein}. Hence, the value of a Bernstein polynomial on one face is determined by basis functions associated with that face only, \note{and} the jumps of polynomial solutions under Bernstein bases across element interfaces can be computed similarly using node-to-node connectivity maps and degrees of freedom corresponding to face points on two neighboring elements.  


Evaluating the DG formulation (\ref{eq:WADGform}) requires applying derivative and lift operators.  These steps can be accelerated using properties of Bernstein polynomials.  Let $\bm{D}^i$ be the Bernstein derivative operator with respect to $i$th barycentric coordinate.  \note{Differentiation matrices with respect to reference coordinates can be expressed as a linear combination of barycentric differentiation matrices $\bm{D}^i$.}  It can be shown that each row of \note{$\bm{D}^i$} has at most $d+1$ non-zeros in $d$ dimensions \cite{ainsworth2011bernstein, kirby2011fast}, such that \note{the sparse application of} barycentric Bernstein differentiation matrices requires only $O(N^d)$ operations.  In contrast, nodal derivative operators are \note{generally} dense matrices \note{of} size $N_p\times N_p$, which require $O(N^{2d})$ operations to apply. 

For a Bernstein lift operator $\bm{L}^f$, it was observed in \cite{chan2017gpu} that \note{$\bm{L}^f$} can be factorized as 
$$\bm{L}^f=\bm{E}^f_L\bm{L}_0,$$
where $\bm{E}^f_L$ is the face reduction matrix and $\bm{L}_0$ is a sparse \note{$N_p^f\times N_p^f$} matrix, where $N_p^f$ is the number of degrees of freedom in the $N$th degree polynomial space on a single face.  Moreover, $\bm{L}_0$ has no more than seven nonzeros per row (independent of $N$).\footnote{\note{Explicit expressions for $\bm{E}^f_L$ and  $\bm{L}_0$ can be found in \cite{chan2017gpu}.  }}  \note{The fixed bandwidth of }the matrix $\bm{L}_0$ \note{implies that it } can be applied \note{in $O(N^{d-1})$ operations}.  The face reduction operator $\bm{E}^f_L$ can be further expanded as product of one-degree reduction operators. Application of $\bm{E}^f_L$ requires applying $N$  triangular one-degree reduction operators, each of which costs $O(N^{d-1})$ to apply. Hence, the total cost of the implementation of the lift matrix $\bm{L}^f$ is $O(N^d)$ in $d$ dimensions. In contrast, the lift matrices under a nodal basis have size $N_p\times N_p^f$ and cost $O(N^{2d-1})$ to apply.

\note{To summarize, the overall cost of evaluating the DG right-hand side is $O(N^d)$ per element in $d$ dimensions under a Bernstein basis.  Since the number of degreees of freedom grows as $O(N^d)$, this complexity is optimal. }  


\section{A fast implementation of weight-adjusted DG methods}\label{sec:BBWADG}

\note{While the evaluation of the BBDG right-hand side requires only  $O(N^d)$ operations per element, this is true only if media is homogeneous (piecewise constant) over each element.  Sub-element heterogeneities can be incorporated using WADG and numerical quadrature as discusssed in Section~\ref{sec:qwadg}. }  However, \note{because quadrature-based WADG involves dense matrix-vector products, the cost generally scales as $O(N^{2d})$ in $d$ dimensions.  Thus, naively utilizing WADG to address sub-cell heterogeneities results in a computational complexity of $O(N^{2d})$ per element, which will dominate the $O(N^d)$ complexity of BBDG and negate any gains in computational efficiency.  }

To address this, we propose a Bernstein-B\'ezier weight-adjusted discontinuous Galerkin (BBWADG) method \note{based on a polynomial approximation of media heterogeneities.  We first note that the evaluation of the DG right-hand side yields a polynomial of degree $N$.  Let $u(x)$ denote this $N$ polynomial, and let $\bm{u}$ denote its coefficients in some basis.  WADG involves applying (\ref{eq:pwadg}) to $\bm{u}$ to compute
\[
\bm{P}_q\textmd{diag}\left(\frac{1}{c^2\left(\bm{\Phi}^k\hat{\bm{x}}\right)}\right)\bm{V}_q\bm{u}.
\]
Since $\bm{P}_q$ is a quadrature-based discretization of the $L^2$ projection operator, this is simply a quadrature-based $L^2$ projection of $u/c^2$ onto polynomials of degree $N$. }

\note{Suppose now that $1/c^2$ is a polynomial of degree $M$.  Then, the main steps of WADG are equivalent to computing $u/c^2$, which is a polynomial of degree $M+N$, and projecting this polynomial onto degree $N$ polynomials.  These two steps correspond to polynomial multiplication and polynomial $L^2$ projection, both of which can be performed efficiently under Bernstein bases.  The resulting algorithms require $O(N^{d+1})$ operations per element in $d$ dimensions.  }  

\note{In practice, we construct a polynomial approximation of $1/c^2$ using a quadrature-based $L^2$ projection of the true wavespeed.  Since the wavespeed does not generally change during a simulation, this approximation can be computed and stored once in a pre-processing step so that it does not affect the computational cost of the solver.}

\note{The remainder of this section describes efficient algorithms for computing the polynomial multiplication and polynomial $L^2$ projection of two Bernstein polynomials.} 
This section is separated into four parts: in Section~\ref{sec:bbmult}, we explain how to compute the product of two Bernstein polynomials as a higher degree Bernstein polynomial.  \note{We introduce Bernstein degree elevation matrices in Section~\ref{sec:bbelev}, which are then used in Section~\ref{sec:bbproj} \note{to construct a representation of the polynomial $L^2$ projection matrix which can be evaluated in $O(N^{d+1})$ operations}.  Finally, we present a GPU-accelerated algorithm of the Bernstein polynomial $L^2$ projection in Section~\ref{sec:bbalgo}.}

\subsection{Bernstein polynomial multiplication}\label{sec:bbmult}

\note{Efficient algorithms exist for the multiplication of two Bernstein polynomials based on discrete convolutions \cite{sanchez2003algebraic}.  We describe a sparse matrix-based implementation here, which is simpler to implement on GPUs.}


Let $B^{N}_{\bm{\alpha}}$ and $B^{M}_{\bm{\beta}}$ be any two Bernstein basis functions of degree $N$ and $M$ respectively. Their product is
\begin{equation*}
\begin{split}
B^{N}_{\bm{\alpha}}B^{M}_{\bm{\beta}}
=\frac{\binom{\bm{\alpha}+\bm{\beta}}{\bm{\alpha}}}{\binom{N+M}{N}}B^{N+M}_{\bm{\alpha}+\bm{\beta}},
\end{split}
\end{equation*}
which is a Bernstein basis function of degree $N+M$ up to a scaling. This observation can be used to efficiently compute the product of two Bernstein polynomials. Let $f({\bm{x}})$ and $g({\bm{x}})$ be two Bernstein polynomials of degree $N$ and $M$ respectively with representations
\begin{equation}
f({\bm{x}})=\sum_{|\bm{\alpha}|=N}a_{\bm{\alpha}}B^N_{\bm{\alpha}}(\bm{x}),\qquad g({\bm{x}})=\sum_{|\bm{\beta}|=M}b_{\bm{\beta}} B^M_{\bm{\beta}}(\bm{x}).
\label{eq:fg}
\end{equation}
Then, $h(\bm{x})=f(\bm{x})g(\bm{x})$ is a Bernstein polynomial of degree $N+M$.  

We first \note{illustrate polynomial multiplication for the} $M=1$ \note{case}, such that $g(\bm{x})$ is a linear polynomial. Let $\bm{e}_j$ denote the canonical vector such that  $g(\bm{x})=\sum_{j=0}^{d}b_jB^1_{\bm{e}_j}(\bm{x})$. Then, the product of $f,g$ is
\begin{equation}
\begin{split}
h(\bm{x})&=\sum_{j=0}^{d}\sum_{|\bm{\alpha}|=N}a_{\bm{\alpha}}b_jB^N_{\bm{\alpha}}(\bm{x})B^1_{\bm{e}_j}(\bm{x})\\
&=\sum_{j=0}^{d}\sum_{|\bm{\alpha}|=N}a_{\bm{\alpha}}b_j\frac{\alpha_{j}+1}{N+1}B^{N+1}_{\bm{\alpha}+\bm{e}_j}(\bm{x}).
\end{split}
\label{eq:hfg}
\end{equation}
Let $\bm{\gamma}$ be a multi-index and $c_{\bm{\gamma}}$ denote the coefficient of $B^{N+1}_{\bm{\gamma}}$ in the expression for $h$ in (\ref{eq:hfg}). Then $c_{\bm{\gamma}}$ can be computed as
\begin{equation}
c_{\bm{\gamma}}=\sum_{j=0}^{d}a_{\bm{\gamma}-\bm{e}_j}b_j\frac{\gamma_{j}}{N+1},
\label{eq:coeff}
\end{equation}
where we set the coefficient to be zero if the corresponding multi-index $\bm{\gamma}-\bm{e}_j$ has negative components.  Hence, for the case $M=1$, the Bernstein coefficients of $h(\bm{x})$ can be expressed as a linear combination of at most $d+1$ products of coefficients for $f(\bm{x})$ and coefficients for $g(\bm{x})$.  This, in turn, can be efficiently computed using sparse matrix operations, as illustrated in Fig.~\ref{fig:polymult}.  
\begin{figure}[!h]
	\centering
	\includegraphics[width=0.8\linewidth]{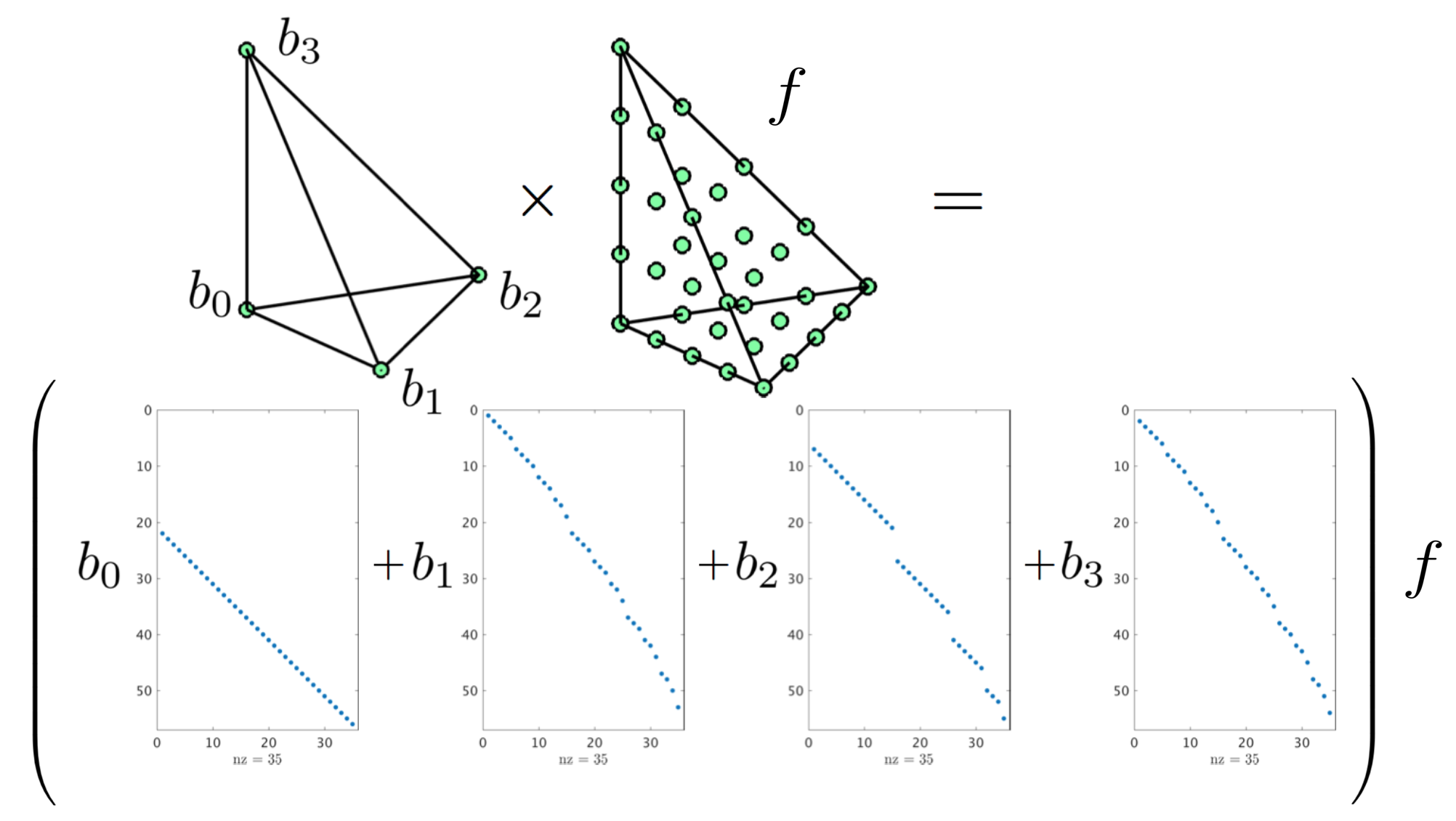}
	\caption[Visualization of Bernstein polynomial multiplication for $M=1$]{Visualization of Bernstein polynomial multiplication for $M=1$}
	\label{fig:polymult}
\end{figure}

We now consider the \note{more general case of arbitrrary $M$}.  We are interested in computing the product $h(\bm{x})=f(\bm{x})g(\bm{x})$, where $f(\bm{x})\in P^N$ and $g(\bm{x})\in P^M$. We have the following
\begin{equation*}
\begin{split}
h(\bm{x}) 
&=\sum_{|\bm{\beta}|=M}\sum_{|\bm{\alpha}|=N}a_{\bm{\alpha}}b_{\bm{\beta}}B^N_{\bm{\alpha}}(\bm{x})B^M_{\bm{\beta}}(\bm{x})\\
&=\sum_{|\bm{\beta}|=M}\sum_{|\bm{\alpha}|=N}a_{\bm{\alpha}}b_{\bm{\beta}}\frac{\binom{\bm{\alpha}+\bm{\beta}}{\bm{\alpha}}}{\binom{N+M}{N}}B^{N+M}_{\bm{\alpha}+\bm{\beta}}(\bm{x}).
\end{split}
\end{equation*}
Hence, the coefficient $c_{\bm{\gamma}}$ of $B^{N+M}_{\bm{\gamma}}$ in $h(\bm{x})$ can be computed as
\begin{equation}
c_{\bm{\gamma}}=\sum_{|\bm{\beta}|=M}a_{\bm{\gamma}-\bm{\beta}}b_{\bm{\beta}}\frac{\binom{\bm{\gamma}}{\bm{\beta}}}{\binom{N+M}{N}}.
\label{eq:mcoeff}
\end{equation}
As in (\ref{eq:coeff}), the coefficient $c_{\bm{\gamma}}$ is set to be zero if the corresponding multi-index $\bm{\gamma}-\bm{\beta}$ has negative components. Hence, $c_{\bm{\gamma}}$ can be written as a combination of at most $M_p$ products of coefficients from $f$ and $h$, where $M_p$ is the dimension of the $M$th degree polynomial space.  \note{As in the $M=1$ case,} the multiplication of two arbitrary Bernstein polynomials can be implemented efficiently using sparse matrix multiplications. 

We can also determine the computational complexity of Bernstein polynomial multiplication from the expression (\ref{eq:mcoeff}) for the product of two Bernstein polynomials.  We summarize this in the following theorem:
\begin{thm}\label{thm:bbmult} 
	The multiplication of two Bernstein polynomials of degree $N$ and $M$ can be performed in $O\left( \left( M N \right)^d\right)$ operations.  For fixed $M$, polynomial multiplication requires $O(N^d)$ operations. 
\end{thm}

\subsection{Bernstein degree elevation operators} 
\label{sec:bbelev}

\note{In this section, we introduce degree elevation matrices, which are used within algorithms for polynomial $L^2$ projection in Section~\ref{sec:bbproj}.}. 
Degree elevation refers to the representation of a lower degree polynomial in a high degree polynomial basis. It can be shown that the $d$-dimensional Bernstein polynomial of degree $N-1$ can be expressed as a linear combination of no more than $d+1$ Bernstein polynomials of degree $N$ \cite{kirby2017fast}. For example, a basis function $B^{N-1}_{\bm{\alpha}}$ can be written as
\begin{equation}
\begin{split}
B^{N-1}_{\bm{\alpha}}
=\sum_{j=0}^{d}\frac{\alpha_j+1}{N}B^{N}_{\bm{\alpha}+\bm{e}_j}, 
\end{split}
\label{eq:bbele}
\end{equation}
where $\bm{e}_j$ is the $j$th canonical vector \cite{kirby2011fast}. This property can be used to construct degree elevation matrices under the Bernstein basis. Let $\bm{E}^N_{N-i}$ denote the degree elevation operator, which evaluates a polynomial of degree $N-i$ as a degree $N$ polynomials on a triangle. From (\ref{eq:bbele}), we know that the one-degree elevation matrix $\bm{E}^N_{N-1}$ is sparse, and only contains at most $d+1$ non-zero entries per row independently of the degree $N$. 

Let $\bm{\alpha}$ denote the multi-index for the row corresponding to the basis function $B^{N-1}_{\bm{\alpha}}$. Then, the non-zero values and column indices $\bm{\beta}$ of $\bm{E}^N_{N-1}$ are
$$\left(\bm{E}^N_{N-1}\right)_{\bm{\alpha},\bm{\beta}}=\frac{\alpha_j+1}{N},\ \ \ \ \bm{\beta}=\bm{\alpha}+\bm{e}_j,\ \ \ j=1,\dots,d.$$
The degree elevation matrix $\bm{E}^N_{N-i}$ between arbitrary degrees can be expressed as the product of one-degree elevation matrices
\begin{equation}
\bm{E}^N_{N-i}=\bm{E}^{N}_{N-1}\bm{E}^{N-1}_{N-2}\cdots\bm{E}^{N-i+1}_{N-i}.
\label{eq:elepro}
\end{equation}
We also refer to the transpose of the degree elevation operator $\left(\bm{E}^N_{N-1}\right)^T$ as the degree reduction operator.


\subsection{Bernstein polynomial $L^2$ projection} \label{sec:bbproj}


\note{Recall that the two steps of BBWADG are polynomial multiplication and polynomial $L^2$ projection.  The first step was discussed in Section~\ref{sec:bbmult}, and we discuss the second step in this section.  We introduce an efficient method of computing the $L^2$ projection of a polynomial to a lower degree polynomial under a Bernstein basis.  This approach is based on a representation of the polynomial projection matrix in terms of sparse one-degree elevation matrices. }

The \note{representation of the polynomial $L^2$ projection matrix using degree elevation matrices} is based on two observations. The first observation is that the polynomial $L^2$ projection operator is rectangular diagonal under a modal (orthogonal) basis.  These modal basis functions \cite{owens1956polynomial,dubiner1991spectral,koornwinder1975two,Proriol1957} are hierarchical and $L^2$ orthogonal, such that 
\begin{singlespace}
	$$\Big(L_{\bm{\gamma}},L_{\bm{\sigma}}\Big)=\begin{cases}
	\text{$1$,}    &{\bm{\gamma}=\bm{\sigma}},\\[2ex]
	\text{$0$,} &{\textmd{otherwise,}}
	\end{cases},\qquad L_{\bm{\gamma}}\in P^{|\bm{\gamma}|},$$
\end{singlespace}
\noindent where $\bm{\gamma}$ and $\bm{\sigma}$ are $d$-dimensional multi-indices. For simplicity, we assume the hierarchical modal basis functions are arranged in ascending order \note{with respect to} $|\bm{\gamma}|$.

The second observation is that the outer product of the degree elevation matrix and its transpose is diagonal under a modal basis. We wish to represent the polynomial $L^2$ projection  matrix as a linear combination of these outer products. We recall some results from \cite{chan2017gpu}, which will be used in this proof. 

\begin{lem}[Lemma A.2 in \cite{chan2017gpu}]\label{lemma:A2}Suppose $p\in P^N(\hat{D})$. Let $\bm{T}$ be the transformation matrix mapping model coefficients to Bernstein coefficients such that 
	$$p=\sum_{|\bm{\gamma}|\leq N}c_{\bm{\gamma}}^LL_{\bm{\gamma}}=\sum_{|\bm{\alpha}|=N}c_{\bm{\alpha}}^BB_{\bm{\alpha}}^N,\ \ \ \ \bm{c}^B=\bm{T}\bm{c}^L,$$
	where $L_{\bm{\gamma}},B_{\bm{\alpha}}^N$ are modal and Bernstein polynomials, respectively. Define $\tilde{\bm{D}}$ as $$\tilde{\bm{D}}=\bm{T}^{-1}_{N-i}\left(\bm{E}^N_{N-i}\right)^T\bm{T}_N$$
	Suppose $0\leq k\leq N$, and let $\lambda_k^N, \lambda_k^{N-i}$ be the distinct eigenvalues of $\bm{M}_N$ and $\bm{M}_{N-i}$, respectively. The entries of $\tilde{\bm{D}}$ are 
	\begin{singlespace}
		$$\tilde{\bm{D}}_{\bm{\nu,\gamma}}=\begin{cases}
		\text{$\lambda_{|\bm{\gamma}|}^{N-i}/\lambda_{|\bm{\gamma}|}^N$,}&{\bm{\nu}=\bm{\gamma}},\\[2ex]
		\text{$0$,} &{\textmd{otherwise,}}
		\end{cases} \ \ \ \ \ \ \tilde{\bm{D}}\in \mathbb{R}^{(N-i)_p,N_p}$$
	\end{singlespace}
	\noindent where $N_p,(N-i)_p$ are the dimensions of the space of polynomials of total degree $N$ and $N-i$, respectively. 
\end{lem}

\begin{cor}[Corollary A.3 in \cite{chan2017gpu}]\label{lemma:A3}Under a transformation to a modal basis, $\bm{E}^N_{N-i}(\bm{E}_{N-i}^N)^T$ is diagonal, with entries
	\begin{singlespace}
		\begin{equation*}
		\left(\bm{T}_N^{-1}\bm{E}^N_{N-i}\left(\bm{E}^N_{N-i}\right)^T\bm{T}_N\right)_{\bm{\gamma},\bm{\gamma}}=\begin{cases}
		\text{$0$,}&{|\bm{\gamma}|>(N-i)},\\[2ex]
		\text{$\lambda_{|\bm{\gamma}|}^{N-i}/\lambda_{|\bm{\gamma}|}^{N}$,}&{|\bm{\gamma}|\leq (N-i)}.
		\end{cases}
		\end{equation*}
	\end{singlespace}
\end{cor}

A straightforward extension of Corollary~\ref{lemma:A3} gives the following corollary:
\begin{cor}\label{lemma:basis}
	Under a transformation to a modal basis, $\left\{\bm{E}^N_{N-i}(\bm{E}_{N-i}^N)^T\right\}^N_{i=0}$ is a basis for any $\bm{D}$ such that
	$$\bm{D}=\begin{pmatrix}
d_0 & & &\\
&d_1\bm{I}_1& &\\
& & \ddots&\\
& & & d_N\bm{I}_N 
\end{pmatrix},$$
where $\bm{I}_i$ is the identity matrix of dimension $(i_p-(i-1)_p)\times(i_p-(i-1)_p)$.
	\end{cor}

Let $\bm{P}^{N+M}_{N}$ denote the Bernstein polynomial $L^2$ projection operator from the polynomial space of degree $N+M$ to the polynomial space of degree $N$. By transforming to a modal basis,  we observe that the projection operator should be a diagonal rectangular matrix with diagonal entries equal to one, i.e.,
$$\bm{T}_N^{-1}\left(\bm{P}^{N+M}_N\right)\bm{T}_{N+M}=
\left(\begin{array}{c|c}
\bm{I}& \bm{0}
\end{array}\right)
$$
where $\bm{T}_N,\bm{T}_{N+M}$ are basis transformation matrices between Bernstein and modal bases of degree $N$ and $N+M$ respectively. Based on this observation, we have the following theorem:
\begin{thm}\label{thm:main}
	There exist $c_j$,  $0\leq j\leq N$, such that 
	\begin{equation}
	\bm{P}^{N+M}_N=\sum_{j=0}^{N} c_j\bm{E}^N_{N-j}\left(\bm{E}^N_{N-j}\right)^T\left(\bm{E}^{N+M}_N\right)^T.
	\end{equation}
	\label{thm:bbinvM}
	 \end{thm}

\begin{proof}
From Lemma~\ref{lemma:A2}, we know that
$$\bm{T}^{-1}_{N}\left(\bm{E}^{N+M}_N\right)^T\bm{T}_{N+M}=\left(\begin{array}{cccc|ccc}
   \frac{\lambda_0^{N}}{\lambda_0^{N\!+\!M}}&& &&0&\cdots&0\\
 &\frac{\lambda_1^{N}}{\lambda_1^{N\!+\!M}}\bm{I}_1& &&\vdots&&\vdots\\
    & &\ddots &&\vdots&&\vdots \\
   && &\frac{\lambda_N^{N}}{\lambda_N^{N\!+\!M}}\bm{I}_N&0&\cdots&0\\
   \end{array}\right),$$
which is a rectangular diagonal diagonal matrix. By Corollary~\ref{lemma:basis}, there exist $c_j,\ 0\leq j\leq N$, such that
$$\sum_{j=0}^{N} c_j\bm{T}^{-1}_N\bm{E}^N_{N-j}\left(\bm{E}^N_{N-j}\right)^T\bm{T}_N=\begin{pmatrix}
\frac{\lambda_0^{N\!+\!M}}{\lambda_0^{N}} & & &\\
&\frac{\lambda_1^{N\!+\!M}}{\lambda_1^{N}}\bm{I}_1& &\\
& & \ddots&\\
& & & \frac{\lambda_N^{N\!+\!M}}{\lambda_N^{N}}\bm{I}_N 
\end{pmatrix}.$$
Then, we obtain
$$\sum_{j=0}^{N} c_j\bm{T}_N^{-1}\bm{E}^N_{N-j}\left(\bm{E}^N_{N-j}\right)^T\left(\bm{E}^{N+M}_N\right)^T\bm{T}_{N+M}=\bm{T}_N^{-1}\left(\bm{P}^{N+M}_N\right)\bm{T}_{N+M}.$$
Multiplying $\bm{T}_N$ and $\bm{T}_{N+M}^{-1}$ from left and right hand side, respectively, gives
\begin{equation}
\bm{P}^{N+M}_N=\sum_{j=0}^{N} c_j\bm{E}^N_{N-j}\left(\bm{E}^N_{N-j}\right)^T\left(\bm{E}^{N+M}_N\right)^T.
\label{eq:decomp}
\end{equation}
\end{proof}

In practice, these coefficient $c_j$ can be computed by solving a linear system.  Table~\ref{tab:1} shows values of $c_j$ for several combinations of degree $N$ and $M$ in three dimensions.
\begin{table}
	\centering
	\begin{tabular}{|c||c|c|c|c|c|c|c|} 
		\hline
		& $c_0$ & $c_1$ & $c_2$ & $c_3$ & $c_4$ & $c_5$ \\
		\hline
		$N=2,\  M=1$  &          $0.6667$  &  -0.0667 &    &   &  &  \\
		\hline
		$N=2,\ M=2$ &    1.0000&    -0.3810 &    0.0238 &    &    &   \\
		\hline
		$N=3,\ M=1$ &   1.6000 &  -0.8000 &   0.1333 &   -0.0048 &   & \\
		\hline
		
		$N=3,\ M=2$ &    1.8182 &    -1.2121 &    0.2273&    -0.0087&    &  \\
		\hline
		$N=4, \ M=1$ &    2.0833 &    -1.5152 &    0.4545 &    -0.0505 &    0.0013 &  \\
		\hline
		$N=4,\ M=2$ &    2.8846 &   -2.7972&    0.9441 &    -0.1119&    0.0029 &  \\
		\hline
		$N=5, \ M=1$ &    2.5714 &    -2.4725 &    1.0989 &    -0.2248 &    0.0180 & -0.0003 \\
		\hline
		$N=5,\ M=2$ &    4.2000 &    -5.3846 &    2.6923 &    -0.5874 &    0.0490 &-0.0009 \\
		\hline
	\end{tabular}  
\caption{Coefficients $c_j$ for the Bernstein polynomial projection matrix $\bm{P}^{M+N}_N$ for different choices of degree $N$ and $M$.}
\label{tab:1}
\end{table}

\subsection{A note on fast mass matrix inversion}\label{sec:bbmass}

It should be noted that the approach described in Theorem~\ref{thm:main} is in fact applicable to matrices beyond the polynomial projection matrix.  For example, since the Bernstein mass matrix is diagonal under a modal basis \cite{kirby2017fast}, the inverse Bernstein mass matrix can also be represented as a combination of degree elevation matrices.  We start with an interesting observation in the proof of Lemma~\ref{lemma:A2} (see \cite{chan2017gpu}):
\begin{lem}\label{lemma:bbmass} Let $\bm{M}_N$ be the Bernstein mass matrix of degree $N$. Under a transformation to a modal basis,  the inverse $\bm{M}^{-1}_N$ is diagonal given by
\begin{equation}
\bm{T}_N^{-1}\bm{M}^{-1}_N\bm{T}_N=\begin{pmatrix}
\frac{1}{\lambda^N_0} & & &\\
&\frac{1}{\lambda^N_1}\bm{I}_1& &\\
& & \ddots&\\
& & & \frac{1}{\lambda^N_N}\bm{I}_N 
\end{pmatrix},
\label{eq:bbmassdecomp}
\end{equation}
where $\lambda^{N}_j$ is the $j$th distinct eigenvalue of $\bm{M}_N$.
\end{lem}
 
Applying Corollary~\ref{lemma:basis} to (\ref{eq:bbmassdecomp}) directly, we obtain the following theorem:
\begin{thm}\label{thm:bbmass} There exist $c_j$, $0\leq j\leq N$, such that, the inverse of Bernstein mass matrix  can be written as
	\begin{equation}
	\bm{M}^{-1}_N= \sum_{j=0}^{N}c_j\bm{E}^N_{N-j}\left(\bm{E}^{N}_{N-j}\right)^T.
	\label{eq:mass}
	\end{equation}
\end{thm}
Using (\ref{eq:mass}), the inverse of a Bernstein mass matrix can be represented as a linear combination of sparse Bernstein degree elevation matrices.  Thus, we can apply $\bm{M}^{-1}_N$ using the expression (\ref{eq:telescope}), which requires $O(N^4)$ operations in 3D.  Since WADG requires only applications of $\bm{V}_q$ and $\bm{P}_q = \bm{M}^{-1}\bm{V}_q^T\bm{W}$, by combining fast mass matrix inversion with efficient $O(N^4)$ algorithms for evaluating Bernstein polynomials at quadrature points \cite{ainsworth2011bernstein}, it is possible to implement 3D quadrature-based WADG in $O(N^4)$ total operations.\footnote{Fast Bernstein mass matrix inversion could also be performed using the algorithm described in \cite{kirby2017fast}.  However, as noted by Kirby, this approach is more involved and may be difficult to implement efficiently on GPUs.}  

In light of these results, one may then ask why we bother with the strategy presented in Section~\ref{sec:BBWADG}, which involves both approximation of the weight function and specialized algorithms for polynomial multiplication and polynomial $L^2$ projection.  The answer lies in the nature of the coefficients $c_j$.  We observe that, when representing the Bernstein mass matrix inverse using (\ref{eq:mass}), the coefficients $c_j$ are highly oscillatory with large positive and negative components (see Table~\ref{tab:2}), which can result in significant numerical roundoff in the application of $\bm{M}^{-1}$ using (\ref{eq:mass}).  In contrast, the coefficients used to represent the Bernstein polynomial projection matrix are much less oscillatory (see Table~\ref{tab:1}) and result in less roundoff error.  

We can estimate sensitivity of \note{(\ref{eq:decomp}) and (\ref{eq:mass})} to roundoff by computing 
\begin{equation}
\sum_{j=0}^N |c_j| 
\label{eq:cond}
\end{equation}
\note{In the context of numerical quadrature with negative weights, the quantity (\ref{eq:cond}) is referred to as the condition number} of a quadrature rule \cite{antil2013two}.  For $N = 7$, the value of (\ref{eq:cond}) is approximately $1.67\times 10^7$ for $\bm{M}^{-1}$.  In contrast, for $N= 7$, the value of (\ref{eq:cond}) for $\bm{P}^{M+N}_N$ is approximately $14.53$ for $M = 1$ and $41.35$ for $M = 2$.  

\note{We also investigated roundoff errors numerically by computing the difference between $\bm{M}^{-1}\bm{b} - \bm{e}$ (where $\bm{M}^{-1}$ is computed using backslash in Matlab) and the quantity 
\[
\sum_{j=0}^{N}c_j\bm{E}^N_{N-j}\left(\bm{E}^{N}_{N-j}\right)^T\bm{b} - \bm{e}.
\]
  Here, $\bm{e}$ is the vector of all ones and $\bm{b} = \bm{M}\bm{e}$.  In the absence of roundoff errors, both quantities should be zero.  However, for all $N$, the roundoff error in applying $\bm{M}^{-1}$ using (\ref{eq:mass}) is larger than the roundoff error incurred when using Matlab's backslash directly.  Since the Bernstein mass matrix $\bm{M}$ is already known to become highly ill-conditioned as $N$ increases \cite{ainsworth2011bernstein, chan2016short}, these numerical experiments suggest that evaluating $\bm{M}^{-1}$ using (\ref{eq:mass}) is impractical for large $N$.}


\begin{table}
	\centering
	\begin{tabular}{|c||c|c|c|c|c|c|} 
		\hline
		& $c_0$ & $c_1$ & $c_2$ & $c_3$& $c_4$&$c_5$\\
		\hline
		$N=1$  &      15  &  -3 &    &   &&  \\
		\hline
		$N=2$ &    157.5&    -90 &  7.5   &    &&   \\
		\hline
		$N=3$ &   1260 &  -1260 & 315   &   -15&&  \\
		\hline
		$N=4$ &   8662.5 &  -12600 &5670   &   -840&26.25&  \\
		\hline
		$N=5$ &   54054 &  -103950 & 69300   &   -18900&1890&-42  \\
		\hline
		\end{tabular}  
\caption{Coefficient $c_j$ for $\bm{M}^{-1}$ represented using (\ref{eq:mass}) for different orders $N$.}
\label{tab:2}
\end{table}

\subsection{GPU algorithms}
\label{sec:bbalgo}

In this section, we describe GPU-accelerated algorithms for Bernstein polynomial multiplication and polynomial $L^2$ projection.  

\subsubsection{Polynomial multiplication}
\note{For polynomial multiplication, we aim to compute Bernstein coefficients of of the product $h(\bm{x})=f(\bm{x})g(\bm{x})$}, where $f(\bm{x}),g(\bm{x})$ are Bernstein polynomials of degree $N$ and degree $M$, respectively.  
From (\ref{eq:mcoeff}), we observe that each coefficient of $h(\bm{x})$ is a linear combination of at most $M_p$ products of coefficients from $f$ and $g$ as follows
\note{
\begin{equation}
\begin{split}
c_{\bm{\gamma}} &= \sum_{|\bm{\beta}|=M}a_{\bm{\gamma}-\bm{\beta}}b_{\bm{\beta}}\frac{\binom{\bm{\gamma}}{\bm{\beta}}}{\binom{N+M}{N}} = \sum_{|\bm{\beta}|=M}a_{\bm{\gamma}-\bm{\beta}}b_{\bm{\beta}} \ell_{\bm{\beta}},
\end{split}
\label{eq:co1}
\end{equation}
where $a_{\bm{\gamma}-\bm{\beta}}$ and $b_{\bm{\beta}}$ are coefficients of $f$ and $g$, respectively.  In our implementation, we store the coefficients $\ell_{\bm{\beta}}$ in some sparse matrix, where the row and column indices correspond to the multi-indices $\bm{\gamma}$ and $\bm{\beta}$, respectively.  Each thread will load non-zero entries in a row of this matrix along with the corresponding coefficients $b_j$ and $a_{\bm{\gamma}-\bm{e}_j}$, compute one of the coefficients $c_{\bm{\gamma}}$, and store the result into shared memory.  }

\subsubsection{Polynomial $L^2$ projection}
\label{sec:polygpu}
\note{We now introduce an algorithm to evaluate the polynomial $L^2$ projection based on (\ref{eq:decomp}).  Unfortunately, it is difficult to directly evaluate (\ref{eq:decomp}) in a low-complexity fashion}.  This is because the degree elevation matrices $\bm{E}^N_{N-j}$ transition from sparse to dense matrices as $j$ increases.  Instead, we evaluate (\ref{eq:decomp}) using an equivalent reformulation.  By plugging (\ref{eq:elepro}) into (\ref{eq:decomp}), we can derive a ``telescoping form'' for $\bm{P}^{M+N}_N$ involving sparse one-degree elevation matrices
\begin{align}
\bm{P}^{N\! +\! M}_N\! &=\! \left(\! c_0\bm{I}\! +\! \bm{E}^N_{N\! -\! 1}\! \left(\! c_1\bm{I}\! +\! \bm{E}^{N\! -\! 1}_{N\! -\! 2}\! \left(\! c_2\bm{I}\! +\! \cdots\! \right)\! \left(\! \bm{E}^{N\! -\! 1}_{N\! -\! 2}\! \right)^T\right)\! \left(\! \bm{E}^N_{N\! -\! 1}\! \right)^T\right)\! \left(\! \bm{E}^{N\! +\! 1}_N\! \right)^T\! \cdots\! \left(\! \bm{E}^{N\! +\! M}_{N\! +\! M\! -\! 1}\! \right)^T \label{eq:telescope} \\
&= \tilde{\bm{P}}_N \left(\! \bm{E}^{N\! +\! 1}_N\! \right)^T\! \cdots\! \left(\! \bm{E}^{N\! +\! M}_{N\! +\! M\! -\! 1}\! \right)^T\nonumber
\end{align}
where we have defined $\tilde{\bm{P}}_N = \left(\! c_0\bm{I}\! +\! \bm{E}^N_{N\! -\! 1}\! \left(\! c_1\bm{I}\! +\! \bm{E}^{N\! -\! 1}_{N\! -\! 2}\! \left(\! c_2\bm{I}\! +\! \cdots\! \right)\! \left(\! \bm{E}^{N\! -\! 1}_{N\! -\! 2}\! \right)^T\right)\! \left(\! \bm{E}^N_{N\! -\! 1}\! \right)^T\right)\!$.  We next provide an algorithm to efficiently evaluate this telescoping expression on GPUs. 

\note{The first step in applying $\bm{P}^{N+M}_N$ is to apply the product of degree reduction matrices $\left(\! \bm{E}^{N\! +\! 1}_N\! \right)^T\! \cdots\! \left(\! \bm{E}^{N\! +\! M}_{N\! +\! M\! -\! 1}\! \right)^T$.  Since each of these matrices is sparse and requires $O(N^d)$ operations to apply, this step has an overall computational complexity of $O(N^d)$ for fixed $M$.}

\note{The next step applies $\tilde{\bm{P}}_N$ to the degree-reduced result.  }We separate the application of $\tilde{\bm{P}}_N$ into two parts.  The first part applies the one-degree reduction matrices in a ``downward'' sweep,  while the second applies the one-degree elevation matrices in an ``upward'' sweep (see Fig.~\ref{fig:bbproj} for an illustration).  Both the application of degree elevation or reduction operators and accumulate results during each step simultaneously.  

\note{We briefly describe our GPU implementation used to apply $\tilde{\bm{P}}_N$.  Let $\bm{p}$ be some vector to which we will apply $\tilde{\bm{P}}_N$. } In the first step, we set $\bm{p}_s=\bm{p}$, then compute the product of $\left(\bm{E}^i_{i-1}\right)^T$ and the matrix-vector product $\bm{p}_s$ stored in shared memory. More specifically, each thread computes the dot product of a sparse row of $\left(\bm{E}^i_{i-1}\right)^T$ with the vector $\bm{p}_s$. The resulting output vector $\bm{q}_s$ will be stored in another shared memory array and transfered to $\bm{p}_s$ after all threads complete their computation. At the same time, $\bm{q}_s$ will be scaled by the constant \note{$c_{\bm{\gamma}}$ in (\ref{eq:co1})} and stored in thread-local register memory. 

For the second part, we compute the product of $\bm{E}^i_{i-1}$ and the vector $\bm{p}_s$ in shared memory, and accumulate results with the values in register memory during each step. More specifically, each thread computes the dot product of a sparse row of $\bm{E}^i_{i-1}$ with $\bm{p}_s$, and the result will be added to the corresponding value in register memory. After the accumulation, the values in register memory will be transfered to $\bm{p}_s$ in shared memory, which will be used in the next step. 

\begin{figure}[!t]
	\centering
	\includegraphics[width=0.75\linewidth]{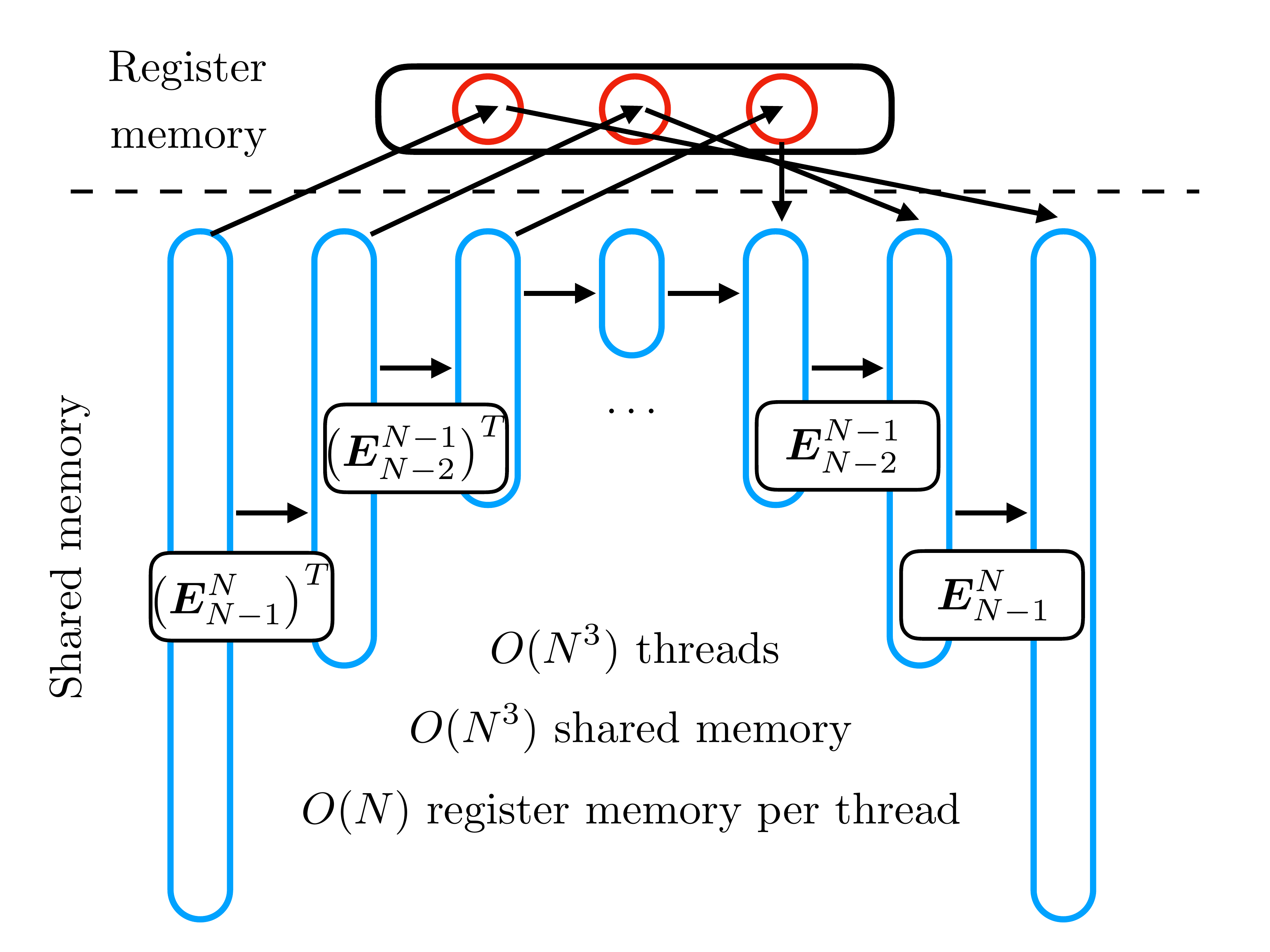}
	\caption[Illustration of GPU algorithm for the polynomial $L^2$ projection]{Illustration of GPU algorithm for the polynomial $L^2$ projection}
	\label{fig:bbproj}
\end{figure}

In our algorithm, the multiplication of two Bernstein polynomials can be \note{computed in $O(N^d)$ operations}. For the polynomial $L^2$ projection, each application of $\bm{E}^i_{i-1}$ or $\left(\bm{E}^i_{i-1}\right)^T$ requires $O(N^d)$ operations. We need to apply $N$ one-degree elevation operators and $N+M$ one-degree reduction operators, resulting in a total asymptotic complexity of $O(N^{d+1})$ \note{for fixed $M$}. This reduces the computational complexity of the projection step in WADG from $O(N^{6})$ to $O(N^{4})$ in three dimensions.

\section{Numerical results}\label{sec:result}In this section, we examine the accuracy and performance of BBWADG.  For \note{clarity}, we refer to WADG as the quadrature-based weight-adjusted discontinuous Galerkin method. This section is divided into four parts: in Section~\ref{sec:converge}, we discuss accuracy of BBWADG using the method of manufactured solutions; In Section~\ref{sec:wavespeed}, we test BBWADG for wavespeed with different frequencies; in Section~\ref{sec:runtime}, we present runtime comparisons between BBWADG and WADG; in Section~\ref{sec:perform}, we present results which quantify the computational efficiency of BBWADG.

\subsection{Convergence for heterogeneous media}\label{sec:converge}
In this section, we investigate the convergence of BBWADG to manufactured solutions. In two dimensions, we assume that the pressure $p(x,y,\tau)$ is of the form
\begin{equation}
p\left(x,y,\tau\right) = \sin\left(\pi x\right)\sin\left(\pi y\right)\cos\left(\pi \tau\right).
\label{eq:pressure2d}
\end{equation}
We take the corresponding velocity vector as follows 
\begin{equation*}
\bm{u}=\begin{pmatrix}
u\\v
\end{pmatrix}=
\begin{pmatrix}
-\cos(\pi x)\sin(\pi y)\sin(\pi \tau)\\
-\sin(\pi x)\cos(\pi y)\sin(\pi \tau)
\end{pmatrix}.
\end{equation*}
Because this is not a solution of the acoustic wave equation in heterogeneous media, we utilize the method of manufactured solutions and add a source term $f(x,y,\tau)$ for which $p(x,y,\tau)$ is a solution.
Plugging $p,\bm{u}$ into (\ref{eq:awave}), we obtain the source term $f$
\begin{equation*}
\begin{split}
f(x,y,\tau) & = \frac{1}{c^2(x,y)}\frac{\partial p}{\partial \tau}+\nabla\cdot \bm{u}\\
&= -\frac{1}{c^2(x,y)}\pi\sin(\pi x)\sin(\pi y)\sin(\pi \tau)+2\pi\sin(\pi x)\sin(\pi y)\sin(\pi \tau)\\
&= \left(2-\frac{1}{c^2(x,y)}\right)\pi\sin(\pi x)\sin(\pi y)\sin(\pi \tau).
\end{split}
\end{equation*}
Similarly, in three dimensions, we assume the pressure $p(x,y,z,\tau)$ satisfies 
$$
p\left(x,y,z,\tau\right) = \sin\left(\pi x\right)\sin\left(\pi y\right)\sin\left(\pi z\right)\cos\left(\pi \tau\right).
$$
We can compute the corresponding velocity vector as follows 
\begin{equation*}
\bm{u}=\begin{pmatrix}
u\\v\\w
\end{pmatrix}=
\begin{pmatrix}
-\cos(\pi x)\sin\left(\pi y\right)\sin(\pi z)\sin(\pi \tau)\\[.5ex]
-\sin(\pi x)\cos\left(\pi y\right)\sin(\pi z)\sin(\pi \tau)\\[.5ex]
-\sin(\pi x)\sin\left(\pi y\right)\cos(\pi z)\sin(\pi \tau)
\end{pmatrix}.
\end{equation*}
Plugging $p,\bm{u}$ into (\ref{eq:awave}), we obtain the source term $f$ 
\begin{equation*}
\begin{split}
f(x,y,z,\tau) & = \frac{1}{c^2(x,y,z)}\frac{\partial p}{\partial \tau}+\nabla\cdot \bm{u}\\
&= \left(3-\frac{1}{c^2(x,y,z)}\right)\pi\sin(\pi x)\sin\left(\pi y\right)\sin(\pi z)\sin(\pi \tau).
\end{split}
\end{equation*}
In numerical experiments, we choose the wavespeed as
\begin{equation*}
c^2(x,y,z) = 1+\frac{1}{2}\sin(\pi x)\sin(\pi y)
\end{equation*}
for two dimensions and 
\begin{equation*}
c^2(x,y,z) = 1+\frac{1}{2}\sin(\pi x)\sin(\pi y)\sin\left(\pi z\right).
\end{equation*}
for three dimensions. In BBWADG, we project $c^2$ onto a polynomial space of degree $M$ in $L^2$ sense. 

Fig.~\ref{fig:con2d} and Fig.~\ref{fig:con3d} show the convergence of BBWADG and WADG to the manufactured solution under mesh refinement. The 3D uniform meshes used in our experiments are generated by GMSH \cite{geuzaine2009gmsh}. From these plots, we observe that the convergence rate of BBWADG is $O(h^r)$, where $r=2$ when $M=0$ and $r=\min\{N+1,M+3\}$ when $M\geq1$. We note that rates of convergence only observed when $c^2$ is approximated using the polynomial $L^2$ projection onto $P^M$. For other approximations (e.g. piecewise linear interpolation), the convergence rates are $O(h^M)$ in general. 

It should be noted that these rates of convergence are better than those suggested by an initial error analysis.  It is straightforward to extend the error analysis of \cite{chan2017weight, chan2018weight} to accomodate approximations of $c^2 \in P^M$.  However, this extension predicts that, when $c^2$ is approximated using $L^2$ projection onto degree $M$ polynomials, the $L^2$ error should converge at a rate of $O(h^{M+1})$.  This rate is observed only for $M=0$, and the source of the discrepancy between the predicted and observed rates for $M> 0$ is presently unclear to the authors.


\note{Increasing from $M=0$ to $M=1$ increases the observed rate of convergence by 2 orders}.  In contrast, increasing $M$ further only increases the observed rate of convergence by one order for each degree past $M=1$.  \note{For this reason, $M=1$ may be an attractive choice for practical computations, since it provides a larger improvement in terms of rates of convergence relative to the increase in computational cost.}  

\begin{figure}[]
	\centering
	\subfloat[Convergence for $N=4$]{
		\begin{tikzpicture}
		\begin{loglogaxis}[
		width=0.48\textwidth,
		height=0.42\textwidth,
		xlabel= Mesh size,
		ylabel= $L^2$ error,
		ticklabel style = {font=\tiny},
		xlabel style={font=\footnotesize},
		ylabel style={font=\footnotesize},
		legend style={font=\tiny},
		legend pos = south east
		]
			\addplot[color=red,mark=square*] coordinates {
				(0.5,0.1575135 )
				(0.25,0.02778378)
				(0.125,0.00701356)
				(0.0625,0.001763432)
				(0.0312,0.0004414745)
			};

			\addplot[color=blue,mark=otimes*] coordinates {
				(0.5,0.00590602 )
				(0.25,0.0004771556)
				(0.125,0.000032725844293)
				(0.0625,0.0000021917583)
				(0.0312,0.000000138034585)
			};
			
			\addplot[color=red,mark=triangle*] coordinates {
				(0.5,0.001672084 )
				(0.25,0.000029571618)
				(0.125,0.000000659821706)
				(0.0625,0.000000019551547)
				(0.0312,0.000000000582975)
			};
			
			\addplot[color=blue,mark=x] coordinates {
				(0.5,0.00080421 )
				(0.25,0.0000167281)
				(0.125,0.00000039029)
				(0.0625,0.0000000109402)
				(0.0312,0.000000000329775504)
			};
			
			\logLogSlopeTriangle{0.2}{0.1}{0.085}{5}{blue};
			\logLogSlopeTriangle{0.2}{0.1}{0.665}{2}{red};
			\logLogSlopeTriangle{0.2}{0.1}{0.335}{4}{blue};
			
			\legend{$M=0$,$M=1$, $M=2$,WADG}
			\end{loglogaxis}
			\end{tikzpicture}
		
	}
	\hskip 2ex
	\subfloat[Convergence for $N=5$]{
		\begin{tikzpicture}
		\begin{loglogaxis}[
		width=0.48\textwidth,
		height=0.42\textwidth,
		title style = {font=\normalsize},
		xlabel= Mesh size,
		ticklabel style = {font=\tiny},
		xlabel style={font=\footnotesize},
		ylabel style={font=\tiny},
		legend style={font=\tiny},
		legend pos = south east
		]
		\addplot[color=red,mark=square*] coordinates {
			(0.5,0.1574265)
			(0.25, 0.027776 )
			(0.125,0.0070135644)
			(0.0625,0.001763477)
			(0.0312,0.000441477)
		};

		\addplot[color=blue,mark=otimes*] coordinates {
			(0.5,0.00581413)
			(0.25, 0.0004768)
			(0.125,0.0000333)
			(0.0625, 0.00000230821)
			(0.0312,0.000000148265632)
		};
		
		\addplot[color=red,mark=triangle*] coordinates {
			(0.5,0.0010914511)
			(0.25, 0.0000197825)
			(0.125,0.00000048003)
			(0.0625, 0.000000015234892)
			(0.0312,0.000000000434665)
		};
		
		\addplot[color=blue,mark=diamond*] coordinates {
			(0.5,0.000268372555936635)
			(0.25,0.000002958830746112)
			(0.125,0.000000036427129667)
			(0.0625,0.000000000587820564)
			(0.0312,0.000000000010006686)
		};

		\addplot[color=red,mark=x] coordinates {
			(0.5,0.0002451 )
			(0.25, 0.0000024406 )
			(0.125,0.0000000282596)
			(0.0625,0.000000000473564581)
			(0.0312,0.000000000008391760)
		};
		
		\logLogSlopeTriangle{0.2}{0.1}{0.705}{2}{red};
		\logLogSlopeTriangle{0.2}{0.1}{0.425}{4}{blue};
		\logLogSlopeTriangle{0.2}{0.1}{0.22}{5}{red};
		\logLogSlopeTriangle{0.2}{0.1}{0.08}{6}{blue};
		\legend{$M=0$,$M=1$, $M=2$,$M=3$,WADG}
		\end{loglogaxis}
		\end{tikzpicture}
	}
	\caption[Convergence to the manufactured solution under mesh refinement (2D)]{Convergence under mesh refinement (2D)}
	\label{fig:con2d}
\end{figure}
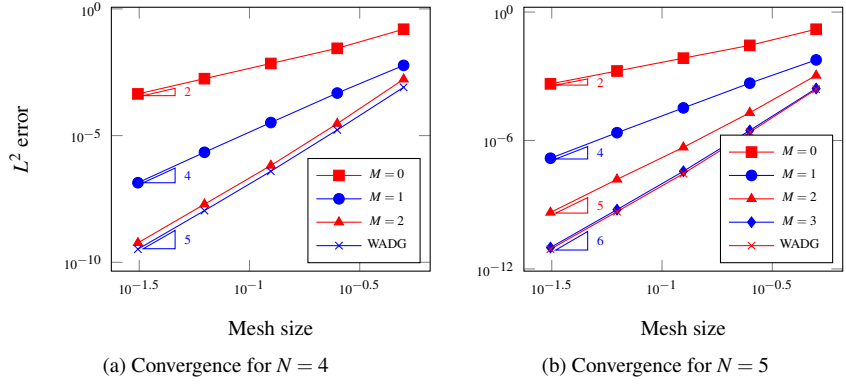

\begin{figure}
	\centering
	\subfloat[Convergence for $N=4$]{
		\begin{tikzpicture}
		\begin{loglogaxis}[
		width=0.48\textwidth,
		height=0.42\textwidth,
		xlabel= Mesh size,
		ylabel= $L^2$ error,
		ticklabel style = {font=\tiny},
			xlabel style={font=\footnotesize},
			ylabel style={font=\footnotesize},
		legend style={font=\tiny},
		legend pos = south east
		]
		\addplot[color=red,mark=square*] coordinates {
			
			(0.5,0.1425)
			(0.25, 0.1181)
			(0.125,0.0267)
			(0.0833,0.012127615353330)
			(0.0625, 0.006888870708627)
			
		};

		\addplot[color=blue,mark=otimes*] coordinates {
			(0.5,0.10194 )
			(0.25,0.007512456)
			(0.125,0.000711573)
			(0.0833,0.00014764)
			(0.0625,0.000047012)
			
		};
		
		\addplot[color=red,mark=triangle*] coordinates {
			(0.5,0.059391879425096 )
			(0.25,0.003057201352859)
			(0.125,0.0000774)
			(0.0833,0.0000095675)
			(0.0625,0.0000021697626)
		};
		
		\addplot[color=blue,mark=x] coordinates {
			(0.5,0.0715124 )
			(0.25,0.0024695)
			(0.125,0.0000733)
			(0.0833,0.0000092115)
			(0.0625,0.00000209026)
		};
		
		\logLogSlopeTriangle{0.18}{0.07}{0.09}{5}{red};
		\logLogSlopeTriangle{0.18}{0.07}{0.685}{2}{red};
		\logLogSlopeTriangle{0.18}{0.07}{0.32}{4}{blue};
		
		\legend{$M=0$,$M=1$, $M=2$,WADG}
		\end{loglogaxis}
		\end{tikzpicture}
		
	}
	\hskip 2ex
\subfloat[Convergence for $N=5$]{
		\begin{tikzpicture}
		\begin{loglogaxis}[
		width=0.48\textwidth,
		height=0.42\textwidth,
		xlabel= Mesh size,
		ticklabel style = {font=\tiny},
		xlabel style={font=\footnotesize},
		ylabel style={font=\tiny},
		legend style={font=\tiny},
		legend pos = south east
		]
		\addplot[color=red,mark=square*] coordinates {
			(0.5,0.184)
			(0.25, 0.118626)
			(0.125,0.026687)
			(0.08333,0.01213)
			(0.0625,0.00688889)
		};

		\addplot[color=blue,mark=otimes*] coordinates {
			(0.5,0.141696594431479)
			(0.25, 0.008195245370250)
			(0.125,0.00071294787)
			(0.0833,0.00014825)
			(0.0625, 0.0000472323)
		};
		
		\addplot[color=red,mark=triangle*] coordinates {
			(0.5,0.05113)
			(0.25, 0.0011656)
			(0.125,0.00002513121)
			(0.08333,0.0000023733)
			(0.0625, 0.0000005026768)
		};
		
		\addplot[color=blue,mark=diamond*] coordinates {
			(0.5,0.048737451042124)
			(0.25,0.00051753)
			(0.125,0.000007582349)
			(0.0833,0.00000059864)
			(0.0625,0.0000001072735)
		};

		\addplot[color=red,mark=x] coordinates {
			(0.5,0.03514 )
			(0.25, 0.00049339)
			(0.125,0.00000724315)
			(0.08333,0.000000573)
			(0.0625,0.00000010271)
		};
		
		\logLogSlopeTriangle{0.18}{0.07}{0.72}{2}{red};
		\logLogSlopeTriangle{0.18}{0.07}{0.44}{4}{blue};
		\logLogSlopeTriangle{0.18}{0.07}{0.185}{5}{red};
		\logLogSlopeTriangle{0.18}{0.07}{0.09}{6}{blue};
		\legend{$M=0$,$M=1$, $M=2$,$M=3$,WADG}
		\end{loglogaxis}
		\end{tikzpicture}
	}
	\caption[Convergence to the manufactured solution under mesh refinement (3D)]{Convergence under mesh refinement (3D).}
	\label{fig:con3d}
\end{figure}
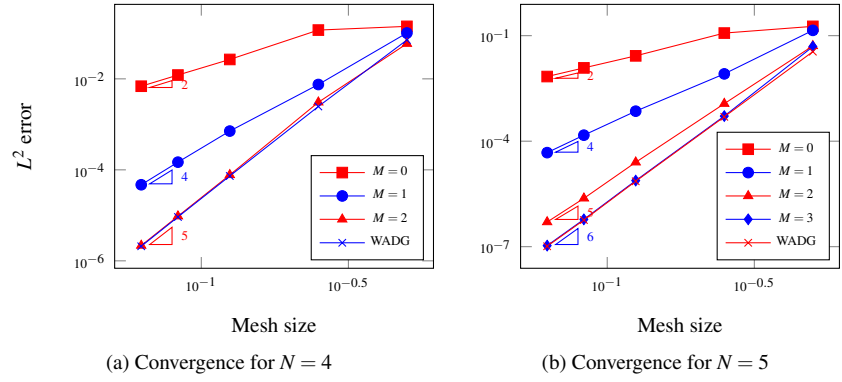

\subsection{Wavespeed with different frequencies}\label{sec:wavespeed}

\note{Since the accuracy of the polynomial approximation of the wavespeed depends on $M$, we examine how the error depends on the approximability of $c^2$.}
We test BBWADG using the following wavespeeds
\begin{equation}
\begin{split}
c^2(x,y) &= 1+\frac{1}{2}\sin(k\pi x)\sin(k\pi y)\qquad \textmd{(2D)},\\
c^2(x,y,z) &= 1+\frac{1}{2}\sin(k\pi x)\sin(k\pi y)\sin(k\pi z)\qquad \textmd{(3D)},
\end{split}
\label{eq:wavespeedk}
\end{equation}
with different frequencies $k$.  \note{However, the manufactured solution remains the same independently of $k$.  } 

\note{This experiment is intended to show how the error depends on the approximability of the wavespeed.  For higher $k$, $c^2$ is more oscillatory and harder to approximate; thus, we expect that the error should increase as $k$ increases, despite the fact that the exact solution is independent of $k$. } 

\begin{figure}
	\centering
	\subfloat[2D convergence]{
		\begin{tikzpicture}
		\begin{semilogyaxis}[
		width=0.48\textwidth,
		height=0.42\textwidth,
		title style = {font=\large},
			xlabel= Degree $M$,
			ylabel= $L^2$ error,
			ticklabel style = {font=\tiny},
			xlabel style={font=\footnotesize},
		    ylabel style={font=\footnotesize},
			legend style={font=\tiny},
			legend pos = south west
			]
			\addplot[color=red,mark=square*] coordinates {
				(0,0.0023828)
				(1,0.00005573332)
				(2,0.0000202349573)
				(3,0.000019581327)
				(4,0.000019602443533)
				(5,0.000019603148488)
				(6,0.0000196031717)
			};
			
			
			
			
			
			\addplot[color=blue,mark=triangle*] coordinates {
				(0,0.005030208127711)
				(1,0.001472601980448)
				(2,0.000560531893102)
				(3,0.000186105423987)
				(4,0.000105283023555)
				(5,0.000109464076391)
				(6,0.000110027427495)
				
			};
			
			\addplot[color=red,mark=otimes*] coordinates {
				(0,0.003751)
				(1,0.0025034)
				(2,0.0017637)
				(3,0.00106317)
				(4,0.00051223336)
				(5,0.00049507257)
				(6,0.0004976484)
				
			};
			
			\addplot[color=blue,mark=diamond*] coordinates {
				(0,0.007091627)
				(1,0.0053486456)
				(2,0.0035956324)
				(3,0.003192266)
				(4,0.00194088267)
				(5,0.0012955886)
				(6,0.0011150673)
				
			};
			\legend{$k=1$,$k=4$,$k=8$,$k=12$}
			\end{semilogyaxis}
		\end{tikzpicture}
		
	}
	\hskip 2ex
	\subfloat[3D convergence]{
		\begin{tikzpicture}
		\begin{semilogyaxis}[
		width=0.48\textwidth,
		height=0.42\textwidth,
		title style = {font=\normalsize},
			xlabel= Degree $M$,
			ticklabel style = {font=\tiny},
			xlabel style={font=\footnotesize},
			ylabel style={font=\tiny},
			legend style={font=\tiny},
			legend pos = south west
		]
			\addplot[color=red,mark=square*] coordinates {
				(0,0.001595283171099 )
				(1,0.000002556506788)
				(2,0.000000018102602)
				(3,0.000000000443356)
				(4,0.000000000003785 )
				(5,0.000000000001831)
				(6,0.000000000001830)
				(7,0.000000000001830)
				
			};
			
			\addplot[color=blue,mark=triangle*] coordinates {
				(0,0.006140566758328 )
				(1,0.000101913479268)
				(2,0.000001986772006)
				(3,0.000000066620022)
				(4,0.000000002739337 )
				(5,0.000000000523401)
				(6,0.000000000504684)
				(7,0.000000000504730)
				
			};
			
			\addplot[color=red,mark=otimes*] coordinates {
				(0,0.023110736206461 )
				(1,0.000983621006701)
				(2,0.000011805156386)
				(3,0.000006014265954)
				(4,0.000000268741911)
				(5,0.000000477608110 )
				(6,0.000000456821622 )
				(7,0.000000457677646)
				
			};
			
			\addplot[color=blue,mark=diamond*] coordinates {
				(0,0.038425099044936)
				(1,0.006702071057236)
				(2,0.000673205091473 )
				(3,0.000057339560940)
				(4,0.000003524433369 )
				(5,0.000001210259573)
				(6,0.000001128494508)
				(7,0.000001126682524)
				
			};
			
			\legend{$k=1$,$k=4$, $k=8$,$k=12$}
			\end{semilogyaxis}
		\end{tikzpicture}
	}
	\caption{Convergence of $L^2$ error when approximating wavespeeds given by (\ref{eq:wavespeedk}).}
	\label{fig:wavespeed}
\end{figure}
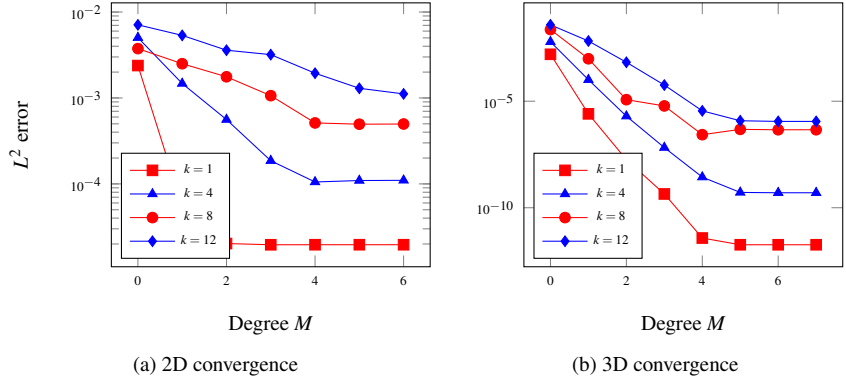

We compute $L^2$ errors on a fixed mesh for various choices of $k$, choose $N=7$ and a uniform mesh with $h=0.0625$ for 2D experiments, and choose $N=6$ and a uniform mesh with $h=0.0833$ for 3D experiments.  From Fig.~\ref{fig:wavespeed}, we observe that, for a fixed $M$, the accuracy of the method \note{does indeed } depend on the frequency of wavespeed: the lower frequency is (or the smaller $k$ is), the smaller the error, despite the fact that the solution remains the same for all $k$.

\subsection{Runtime comparisons}\label{sec:runtime}
In this section, we present runtime comparisons between BBWADG and quadrature-based WADG \note{for $M=1$ and  $M=2$}.  
In Section~\ref{sec:bbalgo}, we showed that the computational complexity of BBWADG is $O(N^{d+1})$ \note{for a fixed $M$}.  In this section, we will verify that this complexity is observed in practice, \note{though the constant depends on $M$}. All results are run on an Nvidia GTX 980 GPU, and the solvers are implemented in the Open Concurrent Compute Abstraction framework (OCCA) \cite{medina2014occa} for clarity and portability.

\subsubsection{Computational implementation}
A time-explicit DG scheme consists of the evaluation of the right hand side and the solution update. Its implementation is typically divided into three kernels. 
\begin{itemize}
	\item A volume kernel, which evaluates contributions to the right hand side resulting from volume terms in (\ref{eq:WADGform}). Specifically, the volume kernel evaluates derivatives of local solutions over each element. 
	
	\item A surface kernel, which evaluates numerical fluxes and contributions to the right hand side resulting from the surface terms in (\ref{eq:WADGform}). More specifically, the surface kernel computes numerical fluxes and applies the lift matrix. 
	
	\item An update kernel, which updates the solution in time. We use a low-storage 4th order Runge-Kutta method \cite{carpenter1994fourth} in this thesis. 
\end{itemize}
We adopt the same volume and surface kernels from \cite{chan2017gpu}. BBWADG and WADG are implemented within the update kernel by modifying the right hand side computed in the volume and surface kernels.

\subsubsection{Acoustic wave equations}
In this experiment, we apply both BBWADG and WADG to the acoustic wave equation (\ref{eq:awave}).  
\note{Runtimes for the update kernels} are given in Fig.~\ref{fig:acoustime}.  \note{The case of $M=1$ is denoted by BBWADG-1, while $M=2$ is denoted by BBWADG-2.} 

\begin{figure}
	\centering
	\begin{tikzpicture}
	\begin{loglogaxis}[
	width=0.55\textwidth,
	height=0.45\textwidth,
	xlabel=Degree $N$,
	ylabel=Runtime (s),
	xlabel style={font=\footnotesize},
	ylabel style={font=\footnotesize},
	xtick = data,
	ticklabel style = {font=\tiny},
	xticklabels={$2$,$3$,$4$,$5$,$6$,$7$,$8$,$9$,$10$},
	legend pos = north west,
	legend style={nodes={scale=0.5, transform shape}}, 
	]
	
	\addplot[color=red,mark=triangle*,mark size=1.5pt] coordinates {
		(2,0.0000000111154)
		(3,0.0000000208683)
		(4,0.0000000331933)
        	        (5,0.0000000656481)
		(6,0.000000085396)
		(7,0.000000135523)
		(8,0.000000165381)
		(9,0.000000266412)
	(10,0.000000405157)
	};

	\addplot[color=red,mark=triangle,mark size=1.5pt] coordinates {
		(2,0.00000002055)
		(3,0.0000000359944)
		(4,0.0000000646931)
		(5,0.0000000967405)
		(6,0.000000150662)
		(7,0.000000196537)
		(8,0.000000284467)
		(9,0.000000416119)
		(10,0.00000056752)
	};
		
	\addplot[dashed, color=red,mark=triangle*,mark size=1.5pt] coordinates {
		(2,0.00000000032)
		(3,0.00000000162)
		(4,0.00000000502)
		(5,0.00000001250)
		(6,0.00000002592)
		(7,0.00000004802)
		(8,0.00000008192)
		(9,0.00000013122)
		(10,0.00000020000)
	};

	\addplot[color=blue,mark=otimes*,mark size=1.5pt] coordinates {
		(2,0.00000000946015)
		(3,0.0000000201553)
		(4,0.0000000490578)
		(5,0.000000120028)
		(6,0.000000219009)
		(7,0.000000487013)
		(8,0.00000525304)
	};
	
	\addplot[dashed, color=blue,mark=diamond*,mark size=1.5pt] coordinates {
		(2,0.00000000032)
		(3,0.00000000364)
		(4,0.00000002048)
		(5,0.00000007812)
		(6,0.00000023328)
		(7,0.00000058824)
		(8,0.00000131072)
		(9,0.00000265720)
	};

	\legend{BBWADG-1,BBWADG-2,$N^4$, WADG,$N^6$}
	\end{loglogaxis}
	\end{tikzpicture}
\caption[Per-element runtimes of update kernels on a mesh of 7854 elements (acoustic)]{Per-element runtimes of update kernels using BBWADG and WADG on a mesh of 7854 elements (acoustic).}
\label{fig:acoustime}
\end{figure}
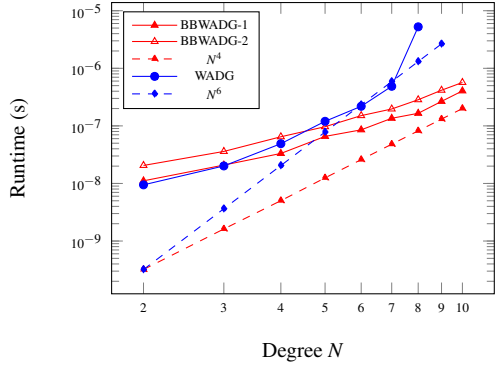

\begin{table}[H]
	\centering
	\begin{tabular}{|c||c|c|c|c|c|c|c|} 
		\hline
		& $N=3$ & $N=4$ & $N=5$ & $N=6$ & $N=7$ & \textcolor{black}{$N=8$} \\
		\hline
		WADG  &          2.02e-8  &  4.91e-8 &   1.20e-7 &  2.19e-7 &  4.87e-7  &  \textcolor{black}{5.25e-6}\\
		\hline
		BBWADG-1 &  2.09e-8  &  3.32e-8 &   6.56e-8 &   8.54e-8 &  1.35e-7 & 1.65e-7 \\
		\hline
		Speedup &    0.9665 &    1.4789 &    1.8292 &    2.5644 &   3.6074 &   \textcolor{black}{31.8182}\\
		\hline
	\end{tabular}  
\caption{Achieved speedup for $M=1$}
\label{tab:am1}
\end{table}
\begin{table}[H]
	\centering
	\begin{tabular}{|c||c|c|c|c|c|c|c|} 
		\hline
		& $N=3$ & $N=4$ & $N=5$ & $N=6$ & $N=7$ & \textcolor{black}{$N=8$} \\
		\hline
		WADG  &          2.02e-8  &  4.91e-8 &   1.20e-7 &  2.19e-7 &  4.87e-7  &  \textcolor{black}{5.25e-6}\\
		\hline
		BBWADG-2 &  3.60e-8  &  6.47e-8 &   9.67e-8 &   1.51e-7 &  1.97e-7 & 2.84e-7 \\
		\hline
		Speedup &    0.5611 &    0.7589 &     1.2409 &    1.4503 & 2.4721 &  \textcolor{black}{18.4859}\\
		\hline
	\end{tabular}  
\caption{Achieved speedup for $M=2$}
\label{tab:am2}
\end{table}

From Fig.~\ref{fig:acoustime}, we observe that BBWADG is more expensive than WADG for low orders $(N < 4)$. However, runtime of the WADG update kernel increases more rapidly with $N$, displaying an asymptotic complexity of $O(N^6)$. On the other hand, the runtime of the BBWADG update kernel increases more slowly and displays a complexity of $O(N^4)$ as proven in Section~\ref{sec:bbalgo}. 

Table~\ref{tab:am1} displays observed speedups of BBWADG over WADG.  We find that for $N=7$, the BBWADG update kernel  for $M=1$ achieves a 3.6 times speedup over the WADG update kernel.  For $N=8$, we observe an unexpected over 30 times speedup.  However, we should note that this result is due to the use of different quadratures between $N=7$ and $N=8$. We choose a tetrahedral quadrature from Xiao and Gimbutas \cite{xiao2010numerical} \note{which is exact for degree $2N+1$ polynomials} for $N\leq 7$. \note{For $N = 8$, this implies that the quadrature rule should be exact for polynomials of degree $17$.  However, the publicly available quadrature rules are only exact up to degree $15$ polynomials. }
Because optimized quadrature points were not publicly available for $N>7$, we \note{switch to} a collapsed coordinate quadrature \cite{Karniadakis.G1999} for $N>7$ (see Fig.~\ref{fig:quadrature}). Since the construction of quadrature points is different, one should not compare results for degrees $N\leq 7$ with degrees $N>7$. 

Table~\ref{tab:am2} shows observed speedups for $M=2$.  We observe that the BBWADG update kernel for $M=2$ is slower than WADG until $N = 5$.  This is due to several reasons. \note{First, increasing from $M=1$ to $M=2$ does not change the overall computational complexity with respect to $N$, but it does change the constant, which scales as $O(M^d)$.}  
\note{Secondly, for $M=1$, since we know \textit{a-priori} that the sparse matrices involved in polynomial multiplication contain only $d+1 = 4$ nonzeros per row in 3D,} we can store such matrices using \verb+float4+ and \verb+int4+ data structures, which have a slightly faster access time on GPUs \cite{chan2017gpu}.  This convenient storage structure is not available for $M=2$. 


\begin{figure}
	\centering
	\subfloat[$N=7$ quadrature]{
		\includegraphics[width=3cm]{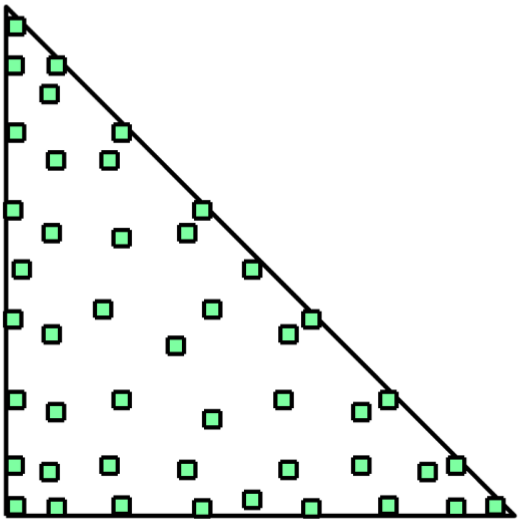}
	}\hskip 5ex 
	\subfloat[$N=8$ quadrature]{
		\includegraphics[width=3cm]{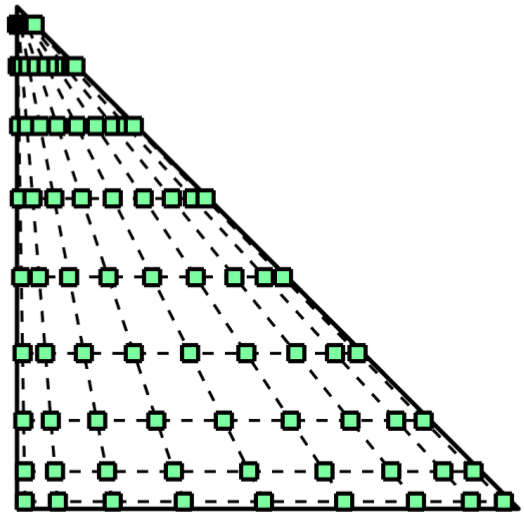}
	}
	\caption[Visualization of the quadrature points on the reference triangle]{Visualization of the quadrature points on the reference triangle}
\label{fig:quadrature}
\end{figure}

\subsubsection{Elastic wave equation}
In this experiment, we compute runtimes for both BBWADG and WADG applied the elastic wave equations (\ref{eq:ewave}).  Computational results for $M=1$ and $M=2$ are presented in Fig.~\ref{fig:elastime}.

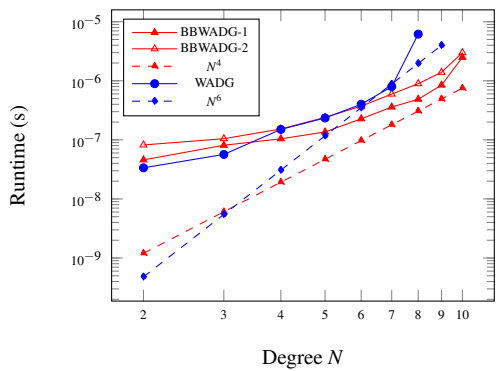
\begin{figure}[]
	\centering
		\begin{tikzpicture}
		\begin{loglogaxis}[
		width=0.55\textwidth,
		height=0.45\textwidth,
		xlabel=Degree $N$,
		ylabel=Runtime (s),
		xlabel style={font=\footnotesize},
		ylabel style={font=\footnotesize},
		xtick = data,
		ticklabel style = {font=\tiny},
		xticklabels={$2$,$3$,$4$,$5$,$6$,$7$,$8$,$9$,$10$},
		legend pos = north west,
	legend style={nodes={scale=0.5, transform shape}}, 
		]

		\addplot[color=red,mark=triangle*,mark size=1.5pt] coordinates {
			(2,0.0000000457983)
		(3,0.0000000809778)
		(4,0.000000104265)
		(5,0.00000013602)
		(6,0.000000227247)
		(7,0.000000359116)
		(8,0.000000487713)
		(9,0.00000084492)
		(10,0.00000249624)
		};
		
		\addplot[color=red,mark=triangle,mark size=1.5pt] coordinates {
		
		(2,0.0000000819964)
		(3,0.000000104635)
		(4,0.000000152827)
		(5,0.000000240527)
		(6,0.000000371963)
		(7,0.000000595837)
		(8,0.000000894296)
		(9,0.00000139124)
		(10, 0.00000300378)
	};
		
		\addplot[dashed, color=red,mark=triangle*,mark size=1.5pt] coordinates {
			(2,0.00000000120)
			(3,0.00000000607)
			(4,0.00000001920)
			(5,0.00000004687)
			(6,0.00000009720)
			(7,0.00000018007)
			(8,0.00000030720)
			(9,0.00000049207)
			(10,0.00000075000)
		};

		\addplot[color=blue,mark=otimes*,mark size=1.5pt] coordinates {
			(2,0.0000000336389)
		(3,0.0000000566972)
		(4,0.000000150267)
		(5,0.000000234543)
		(6,0.000000401095)
		(7,0.000000796524)
		(8,0.00000616693)		};
		
		\addplot[dashed, color=blue,mark=diamond*,mark size=1.5pt] coordinates {
			(2,0.00000000048)
			(3,0.00000000547)
			(4,0.00000003072)
			(5,0.00000011719)
			(6,0.00000034992)
			(7,0.00000088237)
			(8,0.00000196608)
			(9,0.00000398580)
		};

		\legend{BBWADG-1,BBWADG-2,$N^4$, WADG,$N^6$}
		\end{loglogaxis}
		\end{tikzpicture}
	\caption
	{Per-element runtimes for update kernels using BBWADG and WADG on a mesh of 7854 elements (elastic).}
\label{fig:elastime}
\end{figure}

We observe that the runtime behaves similarly to the acoustic case. The runtime of the BBWADG update kernel increases roughly as $O(N^4)$ up to $N = 8$, with about a 2.2 times speedup achieved for $M=1$ and $N=7$.  However, the runtime of BBWADG increases more rapidly than $O(N^4)$ for $N > 8$.  We expect that this is due to GPU occupancy/memory effects.  

\note{The application of $\tilde{\bm{P}}_N$ described in Section~\ref{sec:polygpu} requires storage of $(N+1)$ intermediate values per thread for each application of a scalar weight-adjusted inverse mass matrix.  For the scalar acoustic wave equation, this additional storage is negligible, as only a single weight-adjusted inverse mass matrix is applied per element.  However, for the elastic wave equation, we apply a matrix-weighted weight-adjusted inverse mass matrix, which is computed by applying multiple scalar weight-adjusted inverse mass matrices and combining the results.  For elastic wave propagation in 3D, this increases the per-thread memory cost by a factor of $6$ (corresponding to each of the six components of the elastic stress tensor), resulting in significant register pressure and reduced GPU occupancy.}.

\note{This additional storage can be decreased by processing fewer components simultaneously; however, processing fewer components simultaneously also reduces data reuse and temporal locality.  It is not immediately clear whether this approach will result in an overall lower runtime, and will be the subject of future investigation.}

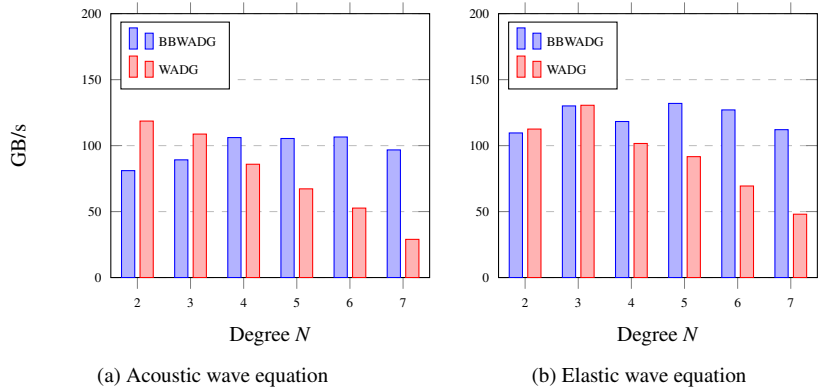
\begin{figure}
	\centering
	\subfloat[Acoustic wave equation]{
		\begin{tikzpicture}
		\begin{axis}[
		width=0.48\textwidth,
		height=0.42\textwidth,
		legend cell align=left,
		xlabel={Degree $N$},
		ylabel={GB/s},
		xmin=1.5, xmax=7.5,
		ymin=0,ymax=200,
		xlabel style={font=\footnotesize},
		ylabel style={font=\footnotesize},
		ticklabel style = {font=\tiny},
		ybar,
		bar width=5pt,
		xtick={2,3,4,5,6,7},
		legend style={font=\tiny},
		legend pos=north west,
		ymajorgrids=true,
		grid style=dashed,
		] 
		\addplot coordinates{(2,81.024)(3,89.2)(4,106.08)(5,105.383)(6,106.517)(7,96.722)};
		\addplot coordinates{(2,118.585)(3,108.734)(4,85.866)(5,67.216)(6,52.638)(7,28.941)};
		\legend{BBWADG, WADG}
		\end{axis}
		\end{tikzpicture}}
	\hskip 2ex
	\subfloat[Elastic wave equation]{
		
		\begin{tikzpicture}
		\begin{axis}[
		width=0.48\textwidth,
		height=0.42\textwidth,
		legend cell align=left,
		xlabel={Degree $N$},
		xmin=1.5, xmax=7.5,
		ymin=0,ymax=200,
		xlabel style={font=\footnotesize},
		ylabel style={font=\tiny},
		ticklabel style = {font=\tiny},
		ybar,
		bar width=5pt,
		xtick={2,3,4,5,6,7},
		legend style={font=\tiny},
		legend pos=north west,
		ymajorgrids=true,
		grid style=dashed,
		] 
		\addplot coordinates{(2,109.584)(3,130.066)(4,118.261)(5,131.968)(6,127.042)(7,112.072)};
		\addplot coordinates{(2,112.53)(3,130.565)(4,101.607)(5,91.625)(6,69.381)(7,48)};

		\legend{BBWADG,WADG}
		\end{axis}
		\end{tikzpicture}
		
	}
	\caption{Achieved bandwidth (GB/s) for update kernels using BBWADG and WADG.}
	\label{fig:bdw}
	
\end{figure}
\subsection{Performance analysis}\label{sec:perform}
In this section, we present computational results for BBWADG with $M=1$ and WADG. Fig.~\ref{fig:bdw} and Fig.~\ref{fig:acp} show the profiled computational performance and bandwidth of the BBWADG and WADG update kernels. From Fig.~\ref{fig:bdw}, we observe that the bandwidth of the WADG update kernel decreases steadily as $N$ increases. In comparison, the BBWADG update kernel sustains a near-constant bandwidth as $N$ increases. 
From Fig.~\ref{fig:acp}, we can see that, for all $N$, the BBWADG update kernel achieves a lower computational performance compared to the WADG update kernel. These results are similar to those achieved for BBDG with piecewise constant wavespeeds \cite{chan2017gpu}.

\begin{figure}
	\centering
	\subfloat[Acoustic wave equation]{
		\begin{tikzpicture}
		\begin{axis}[
		width=0.48\textwidth,
		height=0.42\textwidth,
		legend cell align=left,
		xlabel={Degree $N$},
		ylabel={TFLOPS/s},
		xmin=1.5, xmax=7.5,
		ymin=0,ymax=0.5,
		xlabel style={font=\footnotesize},
		ylabel style={font=\footnotesize},
		ticklabel style = {font=\tiny},
		ybar,
		bar width=5pt,
		xtick={2,3,4,5,6,7},
		legend style={font=\tiny},
		legend pos=north west,
		ymajorgrids=true,
		grid style=dashed,
		] 
		
		\addplot coordinates{(2,0.054676259)(3,0.071267548)(4,0.092135266)(5,0.09482777)(6,0.102449165)(7,0.096437048)};
		\addplot coordinates{(2,0.10036216)(3,0.176223454)(4,0.260597492)(5,0.323972993)(6,0.390148524)(7,0.32094259)};
		\legend{BBWADG, WADG}
		\end{axis}
		\end{tikzpicture}}
	\hskip 2ex
	\subfloat[Elastic wave equation]{
		
		\begin{tikzpicture}
		\begin{axis}[
		width=0.48\textwidth,
		height=0.42\textwidth,
		legend cell align=left,
		xlabel={Degree $N$},
		xmin=1.5, xmax=7.5,
		ymin=0,ymax=2,
		xlabel style={font=\footnotesize},
		ylabel style={font=\tiny},
		ticklabel style = {font=\tiny},
		ybar,
		bar width=5pt,
		xtick={2,3,4,5,6,7},
		legend style={font=\tiny},
		legend pos=north west,
		ymajorgrids=true,
		grid style=dashed,
		] 
		\addplot coordinates{(2,0.20589777)(3,0.258967923)(4,0.253688051)(5,0.298471603)(6,0.30287142)(7,0.28015805)};
		\addplot coordinates{(2,0.24670507)(3,0.540583303)(4,0.791899291)(5,1.176795482)(6,1.35113082)(7,1.341861801)};
		
		\legend{BBWADG,WADG}
		\end{axis}
		\end{tikzpicture}
		
	}
	\caption{Achieved computational performance (TFLOPS/s) for update kernels using BBWADG and WADG.}
	\label{fig:acp}
	
\end{figure}
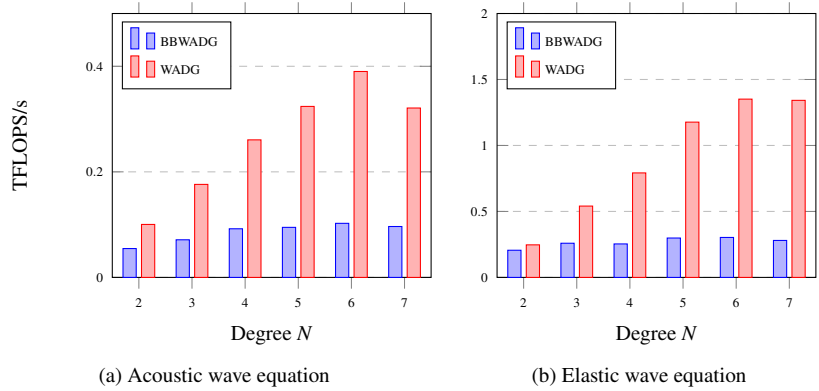

















\section{Conclusion and future work}
In this paper, we present a Bernstein-B\'ezier discontinuous Galerkin (BBWADG) method to simulate acoustic and elastic wave propagation in heterogeneous media \note{based on a polynomial approximation of sub-cell heterogeneities and fast algorithms for Bernstein polynomial multiplication and $L^2$ projection}. \note{ The resulting solver inherits the advantages of WADG (provable energy stability, high order accuracy) while reducing the computational complexity of the update kernel from $O(N^{2d})$ to $O(N^{d+1})$ in $d$ dimensions.  Moreover, this implementation reuses} the BBDG volume and surface kernels from \cite{chan2017gpu}, both of which can be applied in $O(N^d)$ operations.  Thus, the total computational complexity of the BBWADG solver is $O(N^{d+1})$ per timestep \note{for a fixed polynomial approximation of sub-cell media heterogeneities}. 

Due to its low computational complexity, BBWADG offers advantages in simulating wave propagation in heterogeneous media using higher order approximations. These properties make BBWADG promising for accurate and efficient simulation of large-scale wave propagation problems.

\section*{Acknowledgments}
The authors acknowledge the support of the National Science Foundation under awards DMS-1719818 and DMS-1712639.  

\bibliographystyle{model1-num-names}
\bibliography{refs}

\end{document}